\documentclass[10pt,a4paper]{article}

\usepackage[english]{babel}
\usepackage[T1]{fontenc}
\usepackage{bbm}
\usepackage{graphicx}
\usepackage{verbatim}
\usepackage{amsmath}
\usepackage{amssymb}
\usepackage{amsthm}
  \theoremstyle{plain}
  \newtheorem{theorem}{Theorem}
  \newtheorem{lemma}{Lemma}
  \newtheorem{proposition}{Proposition}
  \newtheorem{corollary}{Corollary}
  \newtheorem{assumption}{Assumption}
  \newtheorem{definition}{Definition}
  \newtheorem{example}{Example}
  \newtheorem{remark}{Remark}
\newenvironment{solution}{\begin{proof}[Solution]}{\end{proof}}
\usepackage{mathrsfs}
\usepackage{mathtools}
\usepackage[margin=1in]{geometry}
\usepackage[colorlinks=true, allcolors=blue, hypertexnames=false]{hyperref}
\usepackage{cite}
\usepackage{graphicx}
\usepackage{subcaption}
\usepackage{caption}

\usepackage{tikz}
\usetikzlibrary{arrows, shapes, chains}
\tikzstyle{startstop} = [rectangle, rounded corners, minimum width=3cm, minimum height=1cm, text centered, draw=black, fill=red!30]
\tikzstyle{io} = [trapezium, trapezium left angle = 70, trapezium right angle=110, minimum width=3cm, minimum height=1cm, text centered, draw=black, fill=blue!30]
\tikzstyle{process} = [rectangle, minimum width=5cm, minimum height=2cm, text centered, text width =5cm, draw=black, fill=white]
\tikzstyle{decision} = [diamond, minimum width=3cm, minimum height=1cm, text centered, draw=black, fill=green!30]
\tikzstyle{arrow} = [thick,-&gt;,&gt;=stealth]
\usepackage{wrapfig}

\usepackage{algpseudocode,float,algorithmicx,algorithm}
\makeatletter
\newenvironment{breakablealgorithm}
{
	\begin{center}
		\refstepcounter{algorithm}
		\hrule height.8pt depth0pt \kern2pt
		\renewcommand{\caption}[2][\relax]{
			{\raggedright\textbf{\ALG@name~\thealgorithm} ##2\par}%
			\ifx\relax##1\relax 
			\addcontentsline{loa}{algorithm}{\protect\numberline{\thealgorithm}##2}%
			\else 
			\addcontentsline{loa}{algorithm}{\protect\numberline{\thealgorithm}##1}%
			\fi
			\kern2pt\hrule\kern2pt
		}
	}{
		\kern2pt\hrule\relax
	\end{center}
}
\makeatother

\allowdisplaybreaks[4]
\begin{document}

\title{\textbf{Nonconvex Optimization Framework for Group-Sparse Feedback Linear-Quadratic Optimal Control:\\
Penalty Approach}}

\author{Lechen Feng%
\thanks{Department of Applied Mathematics, The Hong Kong Polytechnic University,
China. Email: {\tt fenglechen0326@163.com}}
\and Xun Li%
\thanks{Department of Applied Mathematics, The Hong Kong Polytechnic University,
China. Email: {\tt li.xun@polyu.edu.hk} This author is supported by RGC of Hong Kong grants 15226922 and 15225124, and partially from 4-ZZVB and PolyU 4-ZZP4.}
\and Yuan-Hua Ni%
\thanks{College of Artificial Intelligence, Nankai University,
Tianjin, China. Email: {\tt yhni@nankai.edu.cn}.}
}

\date{}

\maketitle

\begin{abstract}

This paper develops a unified nonconvex optimization framework for the design of group-sparse feedback controllers in infinite-horizon linear-quadratic (LQ) problems. We address two prominent extensions of the classical LQ problem: the distributed LQ problem with fixed communication topology (DFT-LQ) and the sparse feedback LQ problem (SF-LQ), both of which are motivated by the need for scalable and structure-aware control in large-scale systems. Unlike existing approaches that rely on convex relaxations or are limited to block-diagonal structures, we directly formulate the controller synthesis as a finite-dimensional nonconvex optimization problem with group $\ell_0$-norm regularization, capturing general sparsity patterns. We establish a connection between DFT-LQ and SF-LQ problems, showing that both can be addressed within our unified framework. Furthermore, we propose a penalty-based proximal alternating linearized minimization (PALM) algorithm and provide a rigorous convergence analysis under mild assumptions, overcoming the lack of coercivity in the objective function. The proposed method admits efficient solvers for all subproblems and guarantees global convergence to critical points. Our results fill a key gap in the literature by enabling the direct design of group-sparse feedback gains with theoretical guarantees, without resorting to convex surrogates or restrictive structural assumptions.

\end{abstract}

\section{Introduction}

During the last few decades, distributed control has gained significant attention due to its appealing characteristics, such as the potential to lower communication costs and parallelizable processing.
In distributed control, the global system is managed by multiple individual agents that have limited ability to communicate with only a partial state of the global system.
Among various formulations, distributed linear-quadratic (LQ) optimal control has attracted particular interest as a tractable yet powerful framework for balancing performance and coordination requirements in networked systems, while one of the representative early works in this area is the study by Davison and Ferguson (1981) \cite{kk3}.
Concretely, consider a linear time-invariant (LTI) system
\begin{equation}\label{system}
\begin{aligned}
\dot{x}(t)&=Ax(t)+B_2u(t)+B_1w(t),\\
z(t)&=Cx(t)+Du(t)
\end{aligned}
\end{equation}
with state $x(t)\in\mathbb{R}^n$, input $u(t)\in\mathbb{R}^m$, exogenous disturbance input $w(t)\in\mathbb{R}^l$, controlled output $z(t)\in\mathbb{R}^q$ and matrices $A,B_2,B_1,C,D$ with proper sizes. The infinite-horizon LQ problem is to find a linear static state feedback gain $K$ such that
$u(t)=-Kx(t)$
minimizing the following performance criterion
\begin{equation}\label{J}
J=\int_{0}^{\infty}z(t)^\top z(t){\rm d}t.
\end{equation}
By \cite{b1}, to optimize (\ref{J}), it is equivalent to minimize the $\mathcal{H}_2$-norm of the transfer function
$
H(s) = (C-DK)(sI_n-A+B_2K)^{-1}B_1
$
from $w$ to $z$, and the performance criterion can be reformulated as
\begin{equation*}
  J(K)=\Vert H(s)\Vert_{\mathcal{H}_2}^2={\rm Tr}((C-DK)W_c(C-DK)^\top),
\end{equation*}
where $W_c$ is the controllability Gramian associated with the closed-loop system.
Let $\mathcal{S}$ denote the set of controllers that internally stabilize system (\ref{system}) (see Section 14 of \cite{Zhou-1996} for details). In this paper, we look at solving the following two optimization problems:
\begin{align}
  &\min_{K\in\mathcal{S}}~J(K),~~{\rm s.t.}~~K\in{\mathcal{K}}, \label{intro-1}\\
  &\min_{K\in\mathcal{S}}~J(K),~~{\rm s.t.}~~K\text{~is group sparse}, \label{intro-2}
\end{align}
where $\mathcal{K}\doteq \{K\in\mathbb{R}^{m\times n}\colon K_{ij}=0,~\text{for all}~(i,j)\in\mathbf{U}\}$ with predetermined subset $\mathbf{U}\subseteq [m]\times [n]$ ($[m]\doteq \{1,\dots, m\}, [n]\doteq \{1,\dots, n\}$).
In particular, subspace $\mathcal{K}$ implicitly predefines a communication topology, where the presence of a nonzero entry $K_{ij}$ indicates that state information from subsystem $j$ is available to controller $i$; hence, optimization problem \eqref{intro-1} is known as a distributed LQ problem with fixed communication topology (DFT-LQ problem).
In contrast, problem \eqref{intro-2} promotes an un-predefined group sparse structure, where $K$ is partitioned into $G$ groups $\{K^{(g)}\}_{g=1}^G$, and the sparsity is imposed at the group level, i.e., encouraging many $K^{(g)} = 0$. The readers may see Section \ref{section2} for a more detailed definition of group-sparsity, and in this paper, we name problem \eqref{intro-2} as the sparse feedback LQ problem (SF-LQ problem).

\subsection{Related Works}

\noindent
\textbf{1) Distributed LQ Problem with Fixed Communication Topology.}
DFT-LQ problem presents a significant challenge, while Tsitsiklis and Athans (1985) \cite{Tsitsiklis-1985} point out the inherent difficulty of decentralized decision making and suggest that optimality may be an elusive goal. Therefore, numerous scholars have investigated relaxed formulations of the DFT-LQ problem. Specifically, Rotkowitz and Lall (2006) \cite{rotkowitz2005characterization} define a property called quadratic invariance, showing that it is necessary and sufficient for the constraint set to be preserved under feedback, and that this allows optimal stabilizing decentralized controllers to be synthesized via convex programming. Furieri \textit{et al.} (2020) \cite{furieri2020sparsity} extend the work of Rotkowitz and Lall, and they propose a generalized framework for the convex design of optimal and near-optimal LTI dynamic output feedback controllers with a predetermined sparsity pattern.
However, both of the works mentioned above relax the DFT-LQ problem into an infinite-dimensional convex optimization problem, which necessitates using finite-dimensional approximations (e.g., Ritz approximations). Moreover, higher-order approximations incur extremely high computational complexity, rendering such approximations unsuitable for large-scale problems.

To overcome the limitations of infinite-dimensional optimization, subsequent researchers turn to finite-dimensional formulations of the DFT-LQ problem. An important line focuses exclusively on DFT-LQ problems under a fully decentralized communication topology, i.e.,
\begin{equation}\label{fully-dis}
\begin{aligned}
  \min_{K\in\mathcal{S}}~&J(K),\\
  {\rm s.t.}~& K\in\left\{ K  \,\colon\, K = \mathbf{blockdiag}(K_1, K_2, \dots, K_p),\ K_i \in \mathbb{R}^{m_i \times n_i},~ \sum_{i=1}^p m_i = m,\sum_{i=1}^p n_i = n  \right\}
\end{aligned}
\end{equation}
with prespecified dimensions $p, \{m_i\}_{i=1}^p, \{n_i\}_{i=1}^p$.
Here, $\mathbf{blockdiag}(A_1,\dots,A_m)$ denotes the block diagonal matrix with diagonal entries $A_1,A_2,\dots,A_m$.
Geromel \textit{et al.} (1994) \cite{i14} formulate a convex relaxation of problem \eqref{fully-dis}:
\begin{equation}\label{convex-relaxation}
  \min_{W\in\mathscr{C}}~{\rm Tr}(RW)
\end{equation}
with given $R$ and convex cone $\mathscr{C}$, using a parameter space approach which enables converting problem \eqref{fully-dis} into a convex optimization problem. Building upon the work \cite{i14}, Ma \textit{et al.} (2021) \cite{ref1} design the symmetric Gauss-Seidel semi-proximal augmented Lagrangian method for problem \eqref{convex-relaxation}, while the preserved decentralized structure, robust stability, and robust performance are all suitably guaranteed with high computational effectiveness.
Moreover, Zheng \textit{et al.} (2020) \cite{i9} propose a distributed design method for problem \eqref{convex-relaxation} by exploiting the underlying sparsity properties based on the chordal decomposition of sparse block matrices and ADMM.
Yang \textit{et al.} (2023) \cite{i12} reformulate problem \eqref{convex-relaxation} into a CP (Conic Programming) problem by the block diagonal structure of the gain matrix and the property of the Kronecker product. Then, they propose a proximal ADMM (PADMM) to solve the dual problem, while finding the non-expansive operator of PADMM so that Halpern fixed-point iteration could be used to accelerate the proposed algorithm.

Compared to the fully decentralized problem \eqref{fully-dis}, the DFT-LQ problem allows partial communication among subsystems, which significantly improves the performance while maintaining scalability.
Fazelnia \textit{et al.} (2017) \cite{i10} show that the NP-hard DFT-LQ problem has a quadratic formulation, which can be relaxed to an SDP (Semidefinite Program). If the SDP relaxation has a rank-$1$ solution, a globally optimal distributed controller can be recovered from this solution.
Additionally, Li \textit{et al.} (2021) \cite{i11} consider a distributed reinforcement learning variant for the DFT-LQ problem. They propose a zero-order distributed policy optimization algorithm that learns linear controllers in a distributed fashion, leveraging the ideas of policy gradient, zero-order optimization, and consensus algorithms.
Talebi and Mesbahi (2023) \cite{talebi2022policy} study the second-order geometry of sub-manifolds of Schur stabilizing controllers $\mathcal{S}\cap\mathcal{K}$,  providing a Newton-type algorithm that does not rely on the exponential mapping or traction, while ensuring local convergence guarantees.

\noindent
\textbf{2) Sparse Feedback LQ Problem.}
SF-LQ problem exhibits the following non-constrained reformulation
\begin{equation}\label{l0-intro}
  \min_{K\in\mathcal{S}}~J(K)+\gamma \Vert K\Vert_0
\end{equation}
with constant $\gamma>0$. Due to the nonconvex nature of both the objective function $J(K)$ and the feasible set $\mathcal{S}$, existing sparse optimization algorithms are not directly applicable to problem \eqref{l0-intro}. Therefore, it is necessary to design specialized algorithms tailored to problem \eqref{l0-intro}.
Specifically, Lin \textit{et al.} (2013) \cite{kk4} utilize ADMM to solve problem \eqref{l0-intro}, although no theoretical convergence analysis has been presented.
Bahavarnia (2015) \cite{bahavarnia2015sparse} and Park \textit{et al.} (2020) \cite{park2020structured} investigate the convex relaxation of problem \eqref{l0-intro}:
\begin{equation}\label{l1-intro}
  \min_{K\in\mathcal{S}}~J(K)+\gamma \Vert K\Vert_1,
\end{equation}
and prove that the Structured Policy Iteration algorithm, taking a policy evaluation step and a policy improvement step in an iterative manner, can solve problem \eqref{l1-intro} with a convergence guarantee to a stationary point.
Following the approach of $\ell_1$ relaxation, Takakura and Sato (2024) \cite{takakura2023structured} revisit the policy gradient algorithm based on the gradient projection method and show its global convergence to $\epsilon$-stationary points of the static output sparse feedback LQ problem.
More generally, Negi and Chakrabortty (2020) \cite{negi2020sparsity} study the generalization of problem \eqref{l1-intro} with time-delays in the feedback and fair sharing of bandwidth among users.

\noindent
\textbf{3) First-Order Nonconvex Optimization.} In this paper, both DFT-LQ problem and SF-LQ problem are converted into the following standard form of nonconvex optimization:
\begin{align}
\min_{x,y}~~&f(x,y)\doteq F(x)+G(y)+H(x,y).\label{nonconvex-1}
\end{align}
For unconstrained problem \eqref{nonconvex-1}, all the existing literature makes the following assumptions in order to establish the convergence to a stationary point:
\begin{itemize}
  \item $F, G$ are possibly nonconvex and nonsmooth functions;
  \item The coupling term $H$ is smooth and possibly nonconvex function.
\end{itemize}
 Concretely, Razaviyayn \textit{et al.} (2013) \cite{ref9} study an alternative inexact BCD (Block Coordinate Descent) approach which updates the variable blocks by successively minimizing a sequence of approximations of $f$. Bolte \textit{et al.} (2014) \cite{bolte2014proximal} introduce a PALM (proximal alternating linearized minimization) algorithm for solving problem \eqref{nonconvex-1}. Building on the powerful KL (Kurdyka-{\L}ojasiewicz) property, they derive a self-contained convergence analysis framework and establish that each bounded sequence generated by PALM globally converges to a stationary point.
Furthermore, Bo\c{t} \textit{et al.} (2019) \cite{bot2019proximal} propose a proximal algorithm for minimizing problem \eqref{nonconvex-1} with objective $f$ consisting of three summands: the composition of a nonsmooth function $F$ with a linear operator, another nonsmooth function $G$, and a smooth function $H$ which couples the two block variables. For the improved numerical performance, Pock and Sabach (2016) \cite{pock2016inertial} analyze an inertial version of the PALM algorithm and prove its global convergence to a critical point. Taking this a step further, Gao \textit{et al.} (2020) \cite{gao2020gauss} design the algorithm that is a Gauss-Seidel type inertial PALM (GiPALM) algorithm, and prove that each bounded sequence generated by GiPALM globally converges to a critical point.

\subsection{Contributions}\label{contribution-sec}

The contributions of this paper can be categorized into two aspects: for the LQ problem and nonconvex optimization.

\noindent
\textbf{1) Contributions for LQ Problem.} Based on the above literature review, existing methods for solving DFT-LQ and SF-LQ problems suffer from the following limitations:
\begin{itemize}
  \item The existing methods are specifically tailored for feedback gain with block-diagonal structures, while they lack the flexibility to address feedback gain with general configurations; see \cite{ref1,i12,i9,i14}.
  \item The existing methods involve solving an infinite-dimensional optimization problem, and typically, one needs to employ methods from partial differential equations \cite{gelfand2000calculus}, polynomial approximation \cite{devolder2010solving}; see \cite{furieri2020sparsity, rotkowitz2005characterization, matni2016regularization}.
  \item The existing methods exhibit only a local convergence guarantee or belong to the class of heuristic algorithms; see \cite{i10,kk4,arastoo2016closed}.
  \item For the SF-LQ problem in particular, the existing methods rely on the convex relaxation of the $\ell_0$-norm, e.g., the $\ell_1$-norm; see \cite{bahavarnia2015sparse,park2020structured,takakura2023structured,negi2020sparsity}. However, when feedback gain $K$ is located far from the origin, the $\ell_1$ relaxation may lead to a regime where $\Vert K \Vert_1$ dominates the LQ cost, i.e., $\Vert K \Vert_1 \gg J(K)$; this makes the balance between LQ cost and sparsity no longer hold.
\end{itemize}
In this paper, we directly study the DFT-LQ and SF-LQ problems from the perspective of nonconvex optimization. Concretely, we relax the DFT-LQ and SF-LQ problems into a finite-dimensional problem with group $\ell_0$-norm (rather than the convex relaxation of group $\ell_0$-norm), and provide the global convergence analysis.

\noindent
\textbf{2) Contributions for Nonconvex Optimization.}  In this paper, we study DFT-LQ and SF-LQ problems through nonconvex optimization of the following standard form:
\begin{equation}\label{contri-1}
\begin{aligned}
\min_{\widetilde{W},\widetilde{P}}~~ &f(\widetilde{W})+g(\widetilde{P})\\
{\rm s.t.}~~&\mathcal{A}\widetilde{W}+\mathcal{B}\widetilde{P}=0
\end{aligned}
\end{equation}
with nonsmooth and convex $f$ and group $\ell_0$-norm $g$. First, the linear constraints are introduced into the objective using a penalty function approach, i.e., the following problem is of concern:
\begin{equation}\label{contri-2}
\begin{aligned}
\min_{\widetilde{W},\widetilde{P}}~~ f(\widetilde{W})+g(\widetilde{P})+\frac{\rho}{2}\Vert \mathcal{A}\widetilde{W}+\mathcal{B}\widetilde{P}\Vert^2
\end{aligned}
\end{equation}
with $\rho>0$. We aim to use the PALM method introduced by Bo\c{t} \textit{et al.} \cite{bot2019proximal}; however, the results of \cite{bot2019proximal} cannot be directly used since
\begin{itemize}
  \item there are subproblems within the PALM method, which are hard to solve;
  \item the assumptions of \cite{bot2019proximal} do not hold since both $g$ and $\Vert \mathcal{A}\widetilde{W}+\mathcal{B}\widetilde{P}\Vert^2$ are not coercive; hence, the convergence analysis of \cite{bot2019proximal} cannot be applied in our settings.
\end{itemize}
In this paper, we design an efficient solver for subproblems within the PALM framework, while the convergence analysis and convergence rate are established.

\section{Preliminaries}\label{section2}

{\textbf{Notation}}. Let $\Vert \cdot\Vert$, $\Vert\cdot\Vert_F$, $\Vert\cdot\Vert_0$ and $\Vert \cdot\Vert_{\mathcal{H}_2}$ be the spectral norm, Frobenius norm, $\ell_0$-norm, and $\mathcal{H}_2$-norm of a matrix respectively. $\mathbb{S}^n$ is the set of symmetric matrices of $n\times n$;  $\mathbb{S}^n_{++}$ ($\mathbb{S}^n_+$)  is the set of positive (semi-)definite (PSD) matrices of $n\times n$; $I_n$ is identity matrix of $n\times n$; $\mathbf{0}_n$ is $(0,\dots,0)^\top\in\mathbb{R}^n$ and $\mathbf{1}_n$ is $(1,\dots,1)^\top\in\mathbb{R}^n$.  $A\succ B(A\succeq B)$ means that the matrix $A-B$ is positive (semi-)definite.
{
$\mathrm{vec}(A)$ denotes the column vector formed by stacking columns of $A$ one by one.
$\Gamma_{+}^n$ is $\{\mathrm{vec}(A)\colon A\in\mathbb{S}_+^n\}$.
The operator $\langle A,B\rangle$ denotes the Frobenius inner product, i.e., $\langle A,B\rangle=\mathrm{Tr}(A^\top B)$ for all $A,B\in\mathbb{R}^{m\times n}$, and the notation $\otimes$ denotes the Kronecker product of two matrices.
For $\tau>0$, introduce the proximal operator of $g\colon \mathbb{R}^n\to\mathbb{R}$
\begin{equation*}
  \mathbf{prox}_{\tau g}(z):=\underset{y\in\mathbb{R}^n}{\operatorname*{argmin}}\left\{g(y)+\frac{1}{2\tau}\left\|y-z\right\|^2\right\}.
\end{equation*}
}\noindent
For any subset $C\subseteq\mathbb{R}^n$, $\delta_{C}(\cdot)$ is the indicator function of $C$, i.e.,
\begin{equation*}
  \delta_C(x) = \left\{
  \begin{array}{ll}
    0, & \text{if } x \in C, \\
    +\infty, & \text{if } x \notin C.
  \end{array}
  \right.
\end{equation*}
For $n\geq 1$, the set $[n]$ denotes $\{1,\dots,n\}$. The identity operator, denoted $\mathcal{I}$, is the operator on $\mathbb{R}^n$ such that $\mathcal{I}v=v$, $\forall v\in\mathbb{R}^n$.
Given a vector $x\in\mathbb{R}^n$, the subvector of $x$ composed of the components of $x$ whose indices are in a given subset $\mathbb{I}\subseteq [n]$ is denoted by $x_\mathbb{I}\in\mathbb{R}^{|\mathbb{I}|}$.
For an extended real-valued function $f\colon \mathbb{R}^n\to [-\infty,+\infty]$, the effective domain is the set $\mathrm{dom}(f)=\{x\in\mathbb{R}^n\colon f(x)<\infty\}$.
 Given a sequence $\{a_n\}_{n\geq 0}$, we denote the set of all cluster points of the sequence by $\mathbf{Cluster}(\{a_n\}_{n\geq 0})$. The graph of \( f \), denoted by \( \mathbf{graph}(f) \), is defined as
$
\mathbf{graph}(f) := \left\{ (x, f(x)) \in \mathbb{R}^n \times \mathbb{R}^m \;\middle|\; x \in \mathrm{dom}(f) \right\}.
$
Given positive integer $a,b$, $a \bmod b$  denotes the remainder obtained when $a$ is divided by $b$. Given $x\in\mathbb{R}$, $
\lfloor x \rfloor$ denotes $\max\{n \in \mathbb{Z} : n \leq x\}$.

We now introduce the subdifferential operator for nonconvex functions. Unlike convex functions, the subdifferential of a nonconvex function can be defined in several different ways. Let $\psi\colon\mathbb{R}^n\to\mathbb{R}\cup\{+\infty\}$ be a proper and lower semicontinuous function, and $x\in\mathrm{dom}(\psi)$. The Fr\'{e}chet subdifferential of $\psi$ at $x$ is
\begin{equation*}
  \hat{\partial}\psi(x)\doteq \left\{d\in\mathbb{R}^n\colon \liminf\limits_{y\to x}\frac{\psi(y)-\psi(x)-\langle d,y-x\rangle}{\Vert y-x\Vert}\geq 0\right\},
\end{equation*}
and the limiting (Mordukhovich) subdifferential of $\psi$ at $x$ is
\begin{equation*}
  \partial \psi(x)\doteq \left\{ x\in\mathbb{R}^n\colon \exists \text{~sequences}~x_n\to x,d_n\to d,~\mathrm{s.t.}~\psi(x_n)\to\psi(x),d_n\in\hat{\partial}\psi(x_n)\right\}.
\end{equation*}
For $x\notin \mathrm{dom}\psi$, we set $\hat{\partial}\psi(x)=\partial \psi(x)\doteq \emptyset$. If $\psi$ is convex, then the two subdifferentials coincide with the convex subdifferential of $\psi$, and thus
\begin{equation*}
  \hat{\partial}\psi(x)=\partial \psi(x)=\{d\in\mathbb{R}^n\colon \psi(y)\geq \psi(x)+\langle d,y-x\rangle~\forall y\in\mathbb{R}^n\}.
\end{equation*}
Therefore, using the same notation for the Mordukhovich subdifferential operator of nonconvex functions as for the subdifferential operator of convex functions does not lead to any contradiction. The Mordukhovich subdifferential operator is closed; readers may refer to Chapter 8 of \cite{rockafellar-book} for related results. 

Consider a partitioned matrix
\begin{equation*}
  A=\begin{bmatrix}
      A_{11} & \cdots & A_{1t} \\
      \vdots & \ddots & \vdots \\
      A_{s1} & \cdots & A_{st}
    \end{bmatrix}\in\mathbb{R}^{m\times n},
\end{equation*}
where block $A_{ij}\in\mathbb{R}^{m_i\times n_j}$ for $i=1,\dots,s$ and $j=1,\dots,t$. For notational simplicity, we denote such a partitioned matrix as $A=\mathbf{block}_{s,t}(A_{11},\dots,A_{st})$. To characterize the sparsity pattern of the blocks, we define a mapping $g\colon \mathbb{R}^{m\times n}\to \{0,1\}^{s\times t}$ that maps a partitioned matrix $A$ to a binary matrix indicating the presence of nonzero blocks:
\begin{equation*}
  g(A)_{ij}=\left\{
  \begin{aligned}
  &1,~~A_{ij} \neq 0,\\
  &0,~~A_{ij} = 0.
  \end{aligned}
  \right.
\end{equation*}
Based on this mapping, we define the group $\ell_0$-norm of a partitioned matrix $A$ as $\Vert A\Vert_{s,t; A_{11},\dots,A_{st};0}=\Vert g(A)\Vert_0$, which counts the number of nonzero blocks in $A$. For brevity, we write this as $\Vert A\Vert_{s,t;0}$ when the block structure is clear from the context.

\subsection{SF-LQ Formulation}

We assume the following hypotheses hold, which have also been introduced in \cite{ref1,i11,i14}.
\begin{assumption}\label{ass1}
Let $C^\top D=0,D^\top D\succ 0,C^\top C\succ 0,B_1B_1^\top\succ0, (A,B_2)$ be stabilizable and $(A,C)$ have no unobservable modes on the imaginary axis.
\end{assumption}

\begin{remark}\label{ass_1_explain}
	If $z(t)$ exhibits the following form that is studied in \cite{i9}
\begin{equation*}
  z(t)=\begin{bmatrix}
         Q^{\frac{1}{2}} \\
         0
       \end{bmatrix}x(t)+
       \begin{bmatrix}
         0 \\
         R^{\frac{1}{2}}
       \end{bmatrix}u(t)
\end{equation*}
with $Q,R\succ 0$, then $C^\top C\succ0$, $D^\top D\succ 0$ and $C^\top D=0$ naturally hold.
\end{remark}

By introducing group $\ell_0$-norm into objective function of \eqref{intro-2}, SF-LQ problem can be reformulated by the following optimization problem
\begin{equation}\label{sparse_feedback_1}
	\begin{aligned}
            \min_{K\in\mathbb{R}^{m\times n}}~~& J(K)+\gamma\Vert K\Vert_{s,t;K_{11},\dots,K_{st};0}\\
            {\rm s.t.}~~~~&K\in\mathcal{S},
    \end{aligned}
\end{equation}
where $K_{ij}\in\mathbb{R}^{m_i\times n_j}, i=1,\dots,s, j=1,\dots,t$  represent the block components of the feedback gain matrix, $\mathcal{S}$ denotes the set of stabilizing feedback gains, and $\gamma\geq 0$ is a weighting parameter. However, optimization problem \eqref{sparse_feedback_1} presents several fundamental challenges: the nonconvexity of both LQ cost $J(K)$ and stabilizing feedback set $\mathcal{S}$, as well as the discontinuity of group sparsity measure $\Vert K\Vert_{s,t;0}$.
We employ a classic parameterization strategy \cite{Zhou-1996}  to address these challenges by introducing augmented variables to transform problem \eqref{sparse_feedback_1} into a more tractable form. 
Concretely, we introduce the following notations with $p=m+n$:
\begin{equation*}
  \begin{aligned}
  &F=\begin{bmatrix}
      A & B_2 \\
      0 & 0
    \end{bmatrix}\in\mathbb{R}^{p\times p},~
    G=\begin{bmatrix}
        0  \\
        I_m
      \end{bmatrix}\in\mathbb{R}^{p\times m},\\
  & Q=\begin{bmatrix}
        B_1B_1^\top & 0 \\
        0 & 0
      \end{bmatrix}\in\mathbb{S}^p,~
      R=\begin{bmatrix}
          C^\top C & 0 \\
          0 & D^\top D
        \end{bmatrix}\in\mathbb{S}^p.
  \end{aligned}
\end{equation*}

\begin{assumption}\label{ass2}
The parameter $F$ is unknown but convex-bounded, i.e., $F$ belongs to a polyhedral domain, which is expressed as a convex combination of the extreme matrices: $F=\sum_{i=1}^{M}\xi_iF_i,\xi_i\geq 0,\sum_{i=1}^{M}\xi_i=1$, and $F_i=\begin{bmatrix}
                              A_i & B_{2,i} \\
                              0 & 0
                            \end{bmatrix}\in\mathbb{R}^{p\times p}$
denotes the extreme vertex of the uncertain domain.
\end{assumption}

Let the block matrix be 
\begin{equation*}
W=\begin{bmatrix}
W_1 & W_2\\
W_2^\top & W_3
\end{bmatrix}\in\mathbb{S}^p
\end{equation*}
with $W_1\in\mathbb{S}_{++}^n,W_2\in\mathbb{R}^{n\times m},W_3\in\mathbb{S}^m$. Regard $\Theta_i(W)=F_iW+WF_i^\top+Q$ as a block matrix, i.e.,
\begin{equation*}
\Theta_i(W)=\begin{bmatrix}
\Theta_{i,1}(W) & \Theta_{i,2}(W)\\
\Theta_{i,2}^\top(W) & \Theta_{i,3}(W)
\end{bmatrix}\in\mathbb{S}^p.
\end{equation*}
The following theorem introduces a subset of stabilizing controller gains, and finds an upper bound of performance criterion (\ref{J}).
\begin{theorem}[\!\!\cite{ref1}]\label{thm_para}
  One can define the set
  \begin{align*}
  \mathscr{C}&=\{W\in\mathbb{S}^p:W\succeq0,\Theta_{i,1}(W)\preceq0,\forall i=1,2,\ldots,M\},
  \end{align*}
and let
    $\mathscr{K}=\{K=W_2^\top W_1^{-1}\colon W\in\mathscr{C}\}$.
  Then,
  \begin{enumerate}
    \item $K\in\mathscr{K}$ stabilizes the closed-loop system;
    \item $K\in\mathscr{K}$ gives
    $$\langle R,W\rangle\geq \Vert H_i(s)\Vert_{\mathcal{H}_2}^2,~~i=1,\dots,M,$$
    where $\Vert H_i(s)\Vert_{\mathcal{H}_2}$ represents the $\mathcal{H}_2$-norm with respect to the $i$-th extreme system.
  \end{enumerate}
\end{theorem}

\begin{remark}\label{explain_thm1}
The above theorem is directly quoted from the work of J. Ma \textit{et al.} (2021); see Theorem 1 of \cite{ref1} for details. Their work implicitly assumes that
\begin{equation}\label{imp-ass}
W_1\succ 0,~~\forall~W\in\mathscr{C}.
\end{equation}
This assumption, while not explicitly stated, is standard in the literature and generally holds in practice; see Assumption 2 and Lemma 1 of \cite{Fazel-2021cdc}. In cases where this assumption does not hold, $\mathscr{C}$ takes the form:
\begin{equation*}
  \mathscr{C}=\{W\in\mathbb{S}^p\colon W\succeq0, W_1\succ 0,\Theta_{i,1}(W)\preceq0,\forall i=1,2,\ldots,M\}.
\end{equation*}
For any sufficiently small $\delta>0$, we can consider the alternative set:
\begin{equation*}
  \mathscr{C}^\prime\doteq\{W\in\mathbb{S}^p\colon W\succeq0, W_1-\delta I_n\succeq 0,\Theta_{i,1}(W)\preceq0,\forall i=1,2,\ldots,M\}\subseteq \mathscr{C}.
\end{equation*}
The set $\mathscr{C}^\prime$ differs from $\mathscr{C}$ only by the additional linear cone constraint $W_1-\delta I_n\succeq 0$. Consequently, all theoretical results from both \cite{ref1} and this paper extend naturally to analyses involving $\mathscr{C}^\prime$. Therefore, we also maintain assumption \eqref{imp-ass} throughout this paper.

\end{remark}

\begin{remark}\label{explain_ass12}
When $F$ is known, the classic infinite-horizon LQ problem
is equivalent to the following convex optimization problem \cite{Zhou-1996}
\begin{equation}\label{classical_convex_parameterization}
         \begin{aligned}
           \min_{W\in\mathbb{S}^p}~~&\langle R,W\rangle\\
           {\rm s.t.}~~&W\in\mathscr{C}.
         \end{aligned}
\end{equation}
However, we aim to extend the applicability to scenarios where $F$ is unknown but convex-bounded, as specified in Assumption \ref{ass2}. Notably, LQ systems with unknown but convex-bounded parameters are extensively studied in Linear Parameter Varying (LPV) control, which finds widespread applications across various industrial sectors; see Section 10.2.3 of \cite{sename2013robust} for details. In this case, according to Theorem \ref{thm_para},  optimization problem \eqref{classical_convex_parameterization} no longer directly solves the classic LQ problem, as the objective function in \eqref{classical_convex_parameterization} merely provides an upper bound for the performance criterion $J(K)$.
\end{remark}

Based on Theorem \ref{thm_para}, we reformulate the SF-LQ problem as the following optimization problem:
\begin{equation}\label{sparse_feedback_2}
         \begin{aligned}
           \min\limits_{W\in\mathbb{S}^p \atop K\in\mathbb{R}^{m\times n}}~~&\langle R,W\rangle+\gamma \Vert K\Vert_{s,t;0}\\
           {\rm s.t.}~~&W\in\mathscr{C},\\
           &K=W_2^\top W_1^{-1}.
         \end{aligned}
\end{equation}
Due to the nonlinear manifold constraint $K=W_2^\top W_1^{-1}$, optimization problem \eqref{sparse_feedback_2} exhibits substantial challenges. For any given binary-valued block matrix $X=\mathbf{block}_{s,t}(X_{11},\dots,X_{st})\in\mathbb{R}^{m\times n}$ with $X_{ij}\in\{\mathbf{0}_{m_i\times n_j},\mathbf{1}_{m_i\times n_j}\}$ and denoting
\begin{equation*}
\begin{aligned}
  \mathbf{Sparse}(X)=\{&Y=\mathbf{block}_{s,t}(Y_{11},\dots,Y_{st})\in\mathbb{R}^{m\times n}\colon \text{ if }X_{ij}=\mathbf{0}_{m_i\times n_j}\\
   &\text{ then } Y_{ij}=\mathbf{0}_{m_i\times n_j}, i=1,\dots,s;~j=1,\dots,t\},
\end{aligned}
\end{equation*}
it holds that
\begin{equation*}
  \begin{aligned}
&\text{if~~}W_2^\top \in \mathbf{Sparse}(X) \text{ and } W_1 \in \mathbf{Sparse}(\mathbf{blockdiag}(I_{n_1},\dots,I_{n_t})), \\
&\text{then~~}W_2^\top W_1^{-1} \in \mathbf{Sparse}(X).
\end{aligned}
\end{equation*}
Therefore, when $W_1\in \mathbf{Sparse}(\mathbf{blockdiag}(I_{n_1},\dots,I_{n_t}))$, the feedback gain $K=W_2^\top W_1^{-1}$ inherits the same sparsity pattern as $W_2^\top$, allowing us to relax the minimization of $\Vert K\Vert_{s,t;0}$ to minimizing $\Vert W_2^\top\Vert_{s,t;0}$ while imposing the additional linear constraint $W_1\in \mathbf{Sparse}(\mathbf{blockdiag}(I_{n_1},\dots,I_{n_t}))$. Consequently, optimization problem \eqref{sparse_feedback_2} can be relaxed to the following form
\begin{equation}\label{sparse_feedback_3}
         \begin{aligned}
           \min\limits_{W\in\mathbb{S}^p}~~&\langle R,W\rangle+\gamma \Vert W_2^\top\Vert_0\\
           {\rm s.t.}~~&W\in\mathscr{C},\\
           &W_1\in \mathbf{Sparse}(\mathbf{blockdiag}(I_{n_1},\dots,I_{n_t})).
         \end{aligned}
\end{equation}

Denote $N=t(t-1)/2$ and mappings $\xi:\{1,\dots,N\}\to\{1,\dots,t-1\},\kappa:\{1,\dots,N\}\to\{2,\dots,t\}$
\begin{equation*}
\begin{aligned}
  &\xi(j)=\left\{
  \begin{aligned}
  &1,~~\text{if}~~j\in\{1,\dots,t-1\},\\
  &2,~~\text{if}~~j\in\{t,\dots,2t-3\},\\
  &\vdots\\
  &t-1~~\text{if}~~j\in\{ N\},\\
  \end{aligned}
  \right.~~\kappa(j)=\left\{
  \begin{aligned}
  &2,~~\text{if}~~j\in\{1\},\\
  &3,~~\text{if}~~j\in\{2,t\},\\
  &\vdots\\
  &t,~~\text{if}~~j\in\{t-1,2t-3,3t-6,\dots,N\}.\\
  \end{aligned}
  \right.
\end{aligned}
\end{equation*}
Let constraint matrices be
\begin{equation*}
\begin{aligned}
  &V_{j1}=\begin{bmatrix}
           0_{n_i\times \sum_{p=1}^{i-1}n_p} & I_{n_i} & 0_{n_i\times\left(\sum_{p=i+1}^{t} n_p+m\right)}
         \end{bmatrix},~~i=\xi(j),\\
  &V_{j2}^\top=\begin{bmatrix}
0_{n_i\times \sum_{p=1}^{i-1}n_p} & I_{n_i} & 0_{n_i\times\left(\sum_{p=i+1}^{t} n_p+m\right)}
               \end{bmatrix},~~i=\kappa(j),
\end{aligned}
\end{equation*}
and then problem \eqref{sparse_feedback_3} can be equivalently converted to
\begin{equation}\label{sparse_feedback_4}
\begin{aligned}
  \min_{W,P}~~&\langle R,W\rangle+\gamma \Vert P\Vert_{s,t; P_{11},\dots, P_{st}; 0}\\
  {\rm s.t.}~~&W\in\mathbb{S}^p_+,\\
  &-V_2(F_iW+WF_i^\top+Q)V_2^\top\in\mathbb{S}_+^{n},\forall i=1,\dots,M,\\
  &-V_{j1}WV_{j2}=0,\forall j=1,\dots,N\\
  &V_1WV_2^\top-P=0
\end{aligned}
\end{equation}
with $P_{ij}\in\mathbb{R}^{m_i\times n_j},\gamma>0$, $V_1=[0,I_m],V_2=[I_n,0]$. Denoting
$$\Psi_i=-V_2(F_iW+WF_i^\top+Q)V_2^\top,\forall i=1,\dots,M,$$
optimization problem \eqref{sparse_feedback_4} can be rewritten by the following compact form:
\begin{equation}\label{sparse_feedback_5}
  \begin{aligned}
  \min_{W,P}~~&\langle R,W\rangle+\gamma \Vert P\Vert_{s,t; 0}\\
  {\rm s.t.}~~&\mathcal{G}(W)\in\mathcal{K},\\
  &V_1WV_2^\top-P=0
  \end{aligned}
\end{equation}
with
\begin{equation*}
  \begin{aligned}
&\mathcal{G}(W)=(W,\Psi_1,\dots,\Psi_M,V_{11}WV_{12},\dots,V_{N1}WV_{N2}),\\
&\mathcal{K}=\mathbb{S}_+^p\times\mathbb{S}_+^n\times\cdots\times\mathbb{S}_+^n\times 0\times\cdots\times 0.
  \end{aligned}
\end{equation*}
Let $\widetilde{W}={\rm vec}(W),\widetilde{P}={\rm vec}(P), \widetilde{\Psi}_i={\rm vec}(\Psi_i)$ for $i=1,\dots,M$ and denote $\mathcal{A},\mathcal{B}$ with
\begin{equation*}
  \mathcal{A}=\begin{pmatrix}
  V_{12}^\top\otimes V_{11}\\
  \vdots\\
  V_{N2}^\top\otimes V_{N1}\\
  V_2\otimes V_1
  \end{pmatrix}\doteq \begin{pmatrix}
    \widehat{\mathcal{A}}\\
    V_2\otimes V_1
  \end{pmatrix},~~
  \mathcal{B}=\begin{pmatrix}
  0\\
  \vdots\\
  0\\
  -I_{mn}
  \end{pmatrix},
\end{equation*}
where $\widehat{\mathcal{A}}\in\mathbb{R}^{N^*\times p^2}$ and $N^*$ represents the total number of rows in $\widehat{\mathcal{A}}$ ($N^*$ can be directly computed from the problem dimensions).
Optimization problem \eqref{sparse_feedback_5} can be equivalently formulated by the following standard form:
\begin{equation}\label{sparse_feedback_6}
\begin{aligned}
\min_{\widetilde{W},\widetilde{P}}~~ &f(\widetilde{W})+g(\widetilde{P})\\
{\rm s.t.}~~&\mathcal{A}\widetilde{W}+\mathcal{B}\widetilde{P}=0
\end{aligned}
\end{equation}
with
\begin{align*}
&f(\widetilde{W})=\langle {\rm vec}(R),\widetilde{W}\rangle+\delta_{\Gamma_+^p}(\widetilde{W})+\delta_{\Gamma_+^n}(\widetilde{\Psi}_1)+\cdots+\delta_{\Gamma_+^n}(\widetilde{\Psi}_M),\\
&g(\widetilde{P})=\gamma \Vert \mathrm{vec}^{-1}(\widetilde{P})\Vert_{s,t; 0}.
\end{align*}

\subsection{DFT-LQ Formulation}\label{section_link}

DFT-LQ problems have been extensively studied in the literature due to their numerous advantages, including scalability, robustness, flexibility, and computational efficiency; see \cite{i9,i12,b6,i10,i11} and references. In this subsection, we establish the relationship between SF-LQ and DFT-LQ problems, demonstrating that SF-LQ problems can be viewed as a generalization of DFT-LQ problems. This connection allows us to solve DFT-LQ problems through solving the optimization problem \eqref{sparse_feedback_6}.

Generally, a group-sparse feedback gain $K$ can lead to a multi-agent distributed control system. For example, given a group-sparse feedback gain
  \begin{equation*}
    K=\begin{bmatrix}
      K_{11} & \cdots & K_{1t}\\
      \vdots & \ddots & \vdots\\
      K_{s1} & \cdots & K_{st}
    \end{bmatrix}\in\mathbb{R}^{m\times n},\text{~i.e.,~}K=\mathbf{block}_{s,t}(K_{11},\dots,K_{st})
  \end{equation*}
  with $K_{ij}\in\mathbb{R}^{m_i\times n_j}$ and $\sum_{i=1}^{s} m_i=m,\sum_{j=1}^{t}n_j=n$, it can induce a $p$-agents distributed system.
  Specifically, the global control action is composed of local control actions: $u(t)=[u_1(t)^\top,\dots,u_s(t)^\top]^\top$, where $u_i(t)\in\mathbb{R}^{m_i}$ is the control input of
  agent $i$. At time $t$, agent $i$ directly observes a partial state $x_{\mathcal{I}_i}(t)$, and $\mathcal{I}_i$ is given by
  \begin{equation*}
    \mathcal{I}_i=\bigcup_{j=1,\dots,t,~K_{ij}\neq 0}\left\{n_{j-1}+1,\dots,n_j\right\}\subseteq [n], \text{~with~}n_0=0.
  \end{equation*}
It is worth noting that in the $s$-agent distributed systems studied in \cite{ref1,i12,i9}, the observation sets are disjoint:
\begin{equation*}
  \mathcal{I}_i \cap \mathcal{I}_j =\emptyset,~~\forall i,j\in[s],~i\neq j.
\end{equation*}
In contrast, our framework allows the observation sets $\mathcal{I}_i$ and $\mathcal{I}_j$ to overlap, making it a strict generalization of the problems considered in \cite{ref1,i12,i9}.

Similarly, DFT-LQ problem \eqref{intro-1} can be relaxed into the following form
\begin{equation}\label{fixed}
  \begin{aligned}
  \min_{W,P}~~&\langle R,W\rangle\\
  {\rm s.t.}~~&\mathcal{G}(W)\in\mathcal{K},\\
  &[V_1WV_2^\top]_{ij}=0,~(i,j)\in\mathbf{U}\subseteq [m]\times [n],
  \end{aligned}
\end{equation}
which can be regarded as a modification of problem \eqref{sparse_feedback_4}.
By solving problem \eqref{fixed}, the controller with a fixed communication topology can be obtained.
Furthermore, rather than specifying a fixed communication topology, we may wish to simply constrain that certain agents cannot directly observe specific states, while simultaneously keeping the overall communication cost low through group sparsity. Specifically, suppose we want agent $i_k$ to be unable to directly observe state $x_{j_k}(t)$ for $k=1,\dots,p$, while promoting a group sparse feedback structure. In that case, we can modify problem \eqref{fixed} as follows:
\begin{equation}\label{fixed2}
  \begin{aligned}
  \min_{W,P}~~&\langle R,W\rangle+\gamma \Vert P\Vert_{s,t;0}\\
  {\rm s.t.}~~&\mathcal{G}(W)\in\mathcal{K},\\
  &V_1WV_2^\top-P=0,\\
  &[V_1WV_2^\top]_{i_kj_k}=0,~k=1,\dots,p\\
  \end{aligned}
\end{equation}
with predetermined $(i_k,j_k)\in[m]\times [n]$ for $k=1,\dots,p$. Here, the group sparsity term $\Vert P\Vert_{s,t;0}$ encourages a minimal communication topology while respecting the hard constraints on the forbidden connections. Remarkably, problem \eqref{fixed} can be seen as a relaxation problem of
\begin{equation}\label{fixed3}
  \min_{K\in\mathcal{S}}~J(K)+\gamma \Vert K\Vert_0,~~{\rm s.t.}~~K\in{\mathcal{K}},
\end{equation}
which unifies the DFT-LQ and SF-LQ problems into a single formulation.
To the best of our knowledge, this particular formulation of designing a group sparse feedback gain $K$ while enforcing specific entries $[K]_{i_kj_k}=0$ for $k=1,\dots,p$ represents a novel contribution to the literature.

It is worth emphasizing that problem \eqref{fixed} is a convex optimization problem with semidefinite cone and affine constraints. This problem can be solved efficiently by using the methodology introduced by Feng (2024) \cite{feng2024two}, with details provided in Section 3 of that work. Furthermore, problem \eqref{fixed2} differs from problem \eqref{sparse_feedback_6} only in having additional linear constraints on $W$. Hence, problem \eqref{fixed2} can be rewritten in the following standard form:
\begin{equation}\label{joint}
\begin{aligned}
\min_{\widetilde{W},\widetilde{P}}~~ &f(\widetilde{W})+g(\widetilde{P})\\
{\rm s.t.}~~&\mathcal{A}_{\rm new}\widetilde{W}+\mathcal{B}_{\rm new}\widetilde{P}=0,
\end{aligned}
\end{equation}
where matrices $\mathcal{A}_{\rm new}$ and $\mathcal{B}_{\rm new}$ are simply $\mathcal{A}$ and $\mathcal{B}$ augmented with additional rows to encode the extra linear constraints. Since the only difference between problems \eqref{joint} and \eqref{sparse_feedback_6} is these additional linear constraints in the same form, they are essentially equivalent from both algorithmic and convergence analysis perspectives. Therefore, in the remainder of this paper, we study the DFT-LQ and SF-LQ problems through the lens of problem \eqref{sparse_feedback_6}.

\section{Penalty Function Framework and PALM Algorithm}

Due to the nonconvex and nonsmooth nature of optimization problem \eqref{sparse_feedback_6} with linear constraints, finding a solution is extremely challenging. To address this, we employ a penalty function approach by introducing the following unconstrained optimization problem:
\begin{equation}\label{Pen_1}
  \min_{\widetilde{W},\widetilde{P}}~~ F(\widetilde{W},\widetilde{P})= f(\widetilde{W})+g(\widetilde{P})+H(\widetilde{W},\widetilde{P}),
\end{equation}
where the penalty function $H(\widetilde{W},\widetilde{P})$ takes the form
$$H(\widetilde{W},\widetilde{P})=\frac{\rho}{2}\Vert\mathcal{A}\widetilde{W}+\mathcal{B}\widetilde{P} \Vert^2$$
with penalty parameter $\rho>0$.
As $\rho\to\infty$, optimization problem \eqref{sparse_feedback_6} becomes equivalent to optimization problem \eqref{Pen_1}. In this section, we design an efficient algorithm for solving problem \eqref{Pen_1} and establish its convergence analysis. Let $\mu,\beta,\tau>0$ be parameters with $0<\sigma\leq1$. Given an initial point $(\widetilde{W}_0,\widetilde{P}_0,z_0,u_0)$, we consider the following iterative scheme, known as PALM:
\begin{equation}\label{PAML}
  \begin{aligned}
  &\widetilde{P}_{n+1}\in \mathbf{prox}_{\mu^{-1}g}(\widetilde{P}_n-\mu^{-1}\nabla_{\widetilde{P}}H(\widetilde{W}_n,\widetilde{P}_n)),\\
  &z_{n+1}\in\mathbf{prox}_{\beta^{-1}f}(\widetilde{W}_n+\beta^{-1}u_n),\\
  &\widetilde{W}_{n+1}=\widetilde{W}_n-\tau^{-1}(\nabla_{\widetilde{W}}H(\widetilde{W}_n,\widetilde{P}_{n+1})+u_n+\beta(\widetilde{W}_n-z_{n+1})),\\
  &u_{n+1}=u_n+\sigma \beta (\widetilde{W}_{n+1}-z_{n+1}).
  \end{aligned}
\end{equation}
Algorithm \eqref{PAML} was initially introduced by Bolte \textit{et al.} (2014) \cite{bolte2014proximal}, with subsequent extensions to its applicability developed by B{o}\c{t} \textit{et al.} (2019, 2020) \cite{bot2019proximal,boct2020proximal}. A key feature of the convergence analysis in these works is its reliance on the coercivity of the functions $f$, $g$, and $H$  (a condition explicitly required in Theorem 2.10 of \cite{bot2019proximal}). However, in our optimization problem \eqref{Pen_1}, both functions $g$ and $H$ lack this coercivity property. Consequently, the existing convergence theory for algorithm \eqref{PAML} cannot be directly applied to problem \eqref{Pen_1}, necessitating a new convergence analysis framework.

\subsection{Solver of Subproblem w.r.t. $z_n$}\label{subsubzn}

In this subsection, we will design the solver for subproblem w.r.t. $z_n$:
\begin{equation*}
  z_{n+1}\in\mathbf{prox}_{\beta^{-1}f}(\widetilde{W}_n+\beta^{-1}u_n),
\end{equation*}
which is equivalent to the following convex problem:
\begin{equation}\label{zn-1}
\begin{aligned}
\min_{{\rm svec}(W)}~& \langle \mathrm{vec}(R),D_1\mathrm{svec}(W)\rangle+\frac{\beta}{2}\Vert D_1\mathrm{svec}(\widetilde{W})-\xi_n\Vert^2\\
{\rm s.t.}~~~& {\rm svec}(W)\in\widetilde{\Gamma}_+^p,\\
&{\rm svec}(\Psi_i)\in\widetilde{\Gamma}_+^n,~i=1,\dots,M;
\end{aligned}
\end{equation}
here $\xi_n=\widetilde{W}_n+\beta^{-1}u_n$, $\widetilde{\Gamma}_+^p=\{ {\rm svec}(W)\colon W\in\mathbb{S}_+^p\}$, $\mathrm{svec}(\cdot)$ is the vectorization operator on $\mathbb{S}^n$ defined by
\begin{equation*}
  {\rm svec}(S)=\left[s_{11},s_{12},\cdots,s_{n1},\cdots,s_{nn}\right]^\top\in\mathbb{R}^{\frac{n(n+1)}{2}}
\end{equation*}
for $S=(s_{ij})\in\mathbb{S}^n$, and there exist matrices $T_1,T_2,D_1,D_2$ such that
\begin{equation*}
  \begin{aligned}
  &{\rm svec}(W)=T_1{\rm vec}(W),~~{\rm svec}(\Psi_i)=T_2{\rm vec}(\Psi_i),\\
  &{\rm vec}(W)=D_1{\rm svec}(W),~~{\rm vec}(\Psi_i)=D_2{\rm svec}(\Psi_i).
  \end{aligned}
\end{equation*}
Remarkably, matrices $T_1,T_2,D_1,D_2$ can be explicitly expressed; see Remark 6 of \cite{feng2024two} for details.

\begin{assumption}\label{slater}
  Problem \eqref{zn-1} is strictly feasible, i.e., Slater's condition is satisfied and strong duality holds.
\end{assumption}

\begin{remark}
This assumption is commonly used when solving DFT-LQ and SF-LQ problems via optimization landscape \cite{ref1,i12,i9}. While the general conditions under which Assumption \ref{slater} holds remain an open question. However, Zheng \textit{et al.} (2019) \cite{i9} provide valuable insights from the perspective of block-diagonal strategies when $(A,B_2)$ is known. Let us denote the feasible set of optimization problem \eqref{zn-1} by $D$, and define
$$\widetilde{D}=\{W_2\in\mathbb{R}^{n\times m}\colon W_2=\mathbf{blockdiag}\{W_{2,1},\dots,W_{2,p}\}\},$$
where $W_{2,j}\in\mathbb{R}^{n_j\times m_j}$ for $j=1,\dots,p$ with $\sum_{j=1}^p m_j=m$ and $\sum_{j=1}^p n_j=n$. According to \cite{i9}, $D\cap\widetilde{D}$ is nonempty when system \eqref{system} satisfies any of the following conditions:
\begin{itemize}
\item fully actuated,
\item weakly coupled in terms of topological connections,
\item weakly coupled in terms of dynamical interactions.
\end{itemize}
Under any of these conditions, problem \eqref{zn-1} is feasible. Furthermore, since $B_1B_1^\top\succ 0$ and $W_3$ can be freely chosen to ensure $W\succ 0$ (by the Schur complement lemma), problem \eqref{zn-1} is strictly feasible, thus satisfying Assumption \ref{slater}. Additionally, the feasible set of \eqref{zn-1} in our formulation is larger than those in \cite{ref1,i12,i14,i9}, as we impose no structural constraints on $W_2$.
\end{remark}

\begin{lemma}
The dual problem of optimization problem \eqref{zn-1} exhibits an explicit form of quadratic optimization with affine cone constraints.
\end{lemma}

\begin{proof}
Consider the Lagrangian for problem \eqref{zn-1}, introducing dual variables $X_0$ for the constraint ${\rm svec}(W)\in\widetilde{\Gamma}_+^p$ and $X_i$ for the constraints ${\rm svec}(\Psi_i)\in\widetilde{\Gamma}_+^n$ ($i=1,\dots,M$):
\[
\begin{aligned}
\mathcal{L}(\mathrm{svec}(W); X) =\; & \langle \mathrm{vec}(R), D_1 \mathrm{svec}(W) \rangle + \frac{\beta}{2} \| D_1 \mathrm{svec}(W) - \xi_n \|^2 \\
& - \langle X_0, \mathrm{svec}(W) \rangle + \sum_{i=1}^M \langle X_i, T_2 (V_2 \otimes V_2) \mathrm{vec}(Q) \rangle \\
& + \sum_{i=1}^M \langle X_i, T_2 \left[ (V_2 \otimes V_2 F_i) + (V_2 F_i \otimes V_2) \right] D_1 \mathrm{svec}(W) \rangle.
\end{aligned}
\]
Let us collect all the terms linear in $\mathrm{svec}(W)$ and define
\[
d = D_1^\top \mathrm{vec}(R) - X_0 + D_1^\top \sum_{i=1}^M \left[ (V_2 \otimes V_2 F_i)^\top + (V_2 F_i \otimes V_2)^\top \right] T_2^\top X_i.
\]
Then, the Lagrangian simplifies to
\[
\mathcal{L}(\mathrm{svec}(W); X) = \langle d, \mathrm{svec}(W) \rangle + \frac{\beta}{2} \| D_1 \mathrm{svec}(W) - \xi_n \|^2 + \sum_{i=1}^M \langle X_i, T_2 (V_2 \otimes V_2) \mathrm{vec}(Q) \rangle.
\]
The dual function is obtained by minimizing the Lagrangian with respect to $\mathrm{svec}(W)$:
\[
\Theta(X) = \min_{\mathrm{svec}(W)} \mathcal{L}(\mathrm{svec}(W); X).
\]
This is a strictly convex quadratic function in $\mathrm{svec}(W)$, so the minimum is achieved at the stationary point:
\[
\nabla_{\mathrm{svec}(W)} \mathcal{L} = 0 \implies d + \beta D_1^\top (D_1 \mathrm{svec}(W) - \xi_n) = 0.
\]
Solving for $\mathrm{svec}(W)$, we obtain
\begin{equation}\label{sol-svecW}
\mathrm{svec}(W) = (D_1^\top D_1)^{-1} \left( D_1^\top \xi_n - \frac{1}{\beta} d \right).
\end{equation}

Substituting this back into the Lagrangian yields the dual function:
\[
\begin{aligned}
\Theta(X) =\; & \sum_{i=1}^M \langle X_i, T_2 (V_2 \otimes V_2) \mathrm{vec}(Q) \rangle - \frac{1}{2\beta} \langle d, (D_1^\top D_1)^{-1} d \rangle \\
&+ \langle d, (D_1^\top D_1)^{-1} D_1^\top \xi_n \rangle  + \frac{\beta}{2} \| D_1 (D_1^\top D_1)^{-1} D_1^\top \xi_n - \xi_n \|^2.
\end{aligned}
\]
Note that the last term is independent of the dual variables $X$.
Therefore, the dual problem is:
\begin{equation}\label{dual1}
\begin{aligned}
\max_{X_0, X_1, \dots, X_M} \quad & \Theta(X) \\
\text{subject to} \quad & X_0 \in \widetilde{\Gamma}_+^p, \\
& X_i \in \widetilde{\Gamma}_+^n, \quad i=1,\dots,M.
\end{aligned}
\end{equation}
Since $d$ is affine in $X$, the dual objective is a quadratic function of $X$, and the constraints are semidefinite cones. Thus, the dual problem is a quadratic optimization problem with semidefinite cone constraints.
\end{proof}

Letting
\begin{equation*}
\begin{aligned}
  \Omega(X_0,X_1,\dots,X_M)=&-\langle X_1+\dots+X_M,T_2(V_2\otimes V_2){\rm vec}(Q)\rangle\\
  &-\langle \mathcal{L}^\star(X_0,X_1,\dots,X_M),(D_1^\top D_1)^{-1}D_1^\top\xi_n \rangle\\
  &+\frac{1}{2\beta}\Vert (D_1^\top D_1)^{-\frac{1}{2}}\mathcal{L}^\star(X_0,X_1,\dots,X_M)\Vert^2\\
  &+\delta_{\widetilde{\Gamma}_+^p}(X_0)+\delta_{\widetilde{\Gamma}_+^n}(X_1)+\dots+\delta_{\widetilde{\Gamma}_+^n}(X_M),
\end{aligned}
\end{equation*}
then optimization problem \eqref{dual1} can be rewritten in the following compact form
\begin{equation}\label{dual2}
    \min_{X=(X_0\dots,X_M)}~\Omega(X_0,X_1,\dots,X_M).
\end{equation}

In what follows, we focus on solving optimization problem \eqref{dual2} using the proximal-BCD (pBCD) algorithm. The algorithm follows an iterative update scheme:
\begin{equation}\label{update-pbcd}
  X_i^{k+1}\in \underset{X_i}{\operatorname*{argmin}}~ \Omega(X_0^{k+1},X_1^{k+1},\dots,X_{i-1}^{k+1},X_i,X_{i+1}^{k},\dots,X_M^{k})+\frac{1}{2}\left\Vert X_i-X_i^{k}\right\Vert_{\mathcal{S}_i}^2,~~i=0,\dots,M
\end{equation}
with given initial point $X^0=(X_0^0,X_1^0,\dots,X_M^0)$, where $\{\mathcal{S}_i\}_{i=0}^{M}$ is a sequence of given positive linear operators.
Remarkably, the pBCD update rule \eqref{update-pbcd} involves nonsmooth subproblems, which often necessitates the design of iterative algorithms (e.g., subgradient descent method). This substantially increases the computational complexity of the pBCD algorithm. However, due to the special structure of Lagrange dual function $\Theta$ and PSD cone $\mathbb{S}^n_+$, we may choose a special sequence of operators $\{\mathcal{S}_i\}_{i=0}^{M}$ such that pBCD update scheme \eqref{update-pbcd} admits a closed-form solution.

\begin{lemma}
For problem \eqref{dual2}, there exists a sequence of positive linear operators $\{\mathcal{S}_i\}_{i=0}^{M}$ such that when applying the pBCD algorithm, each subproblem
\begin{equation*}
  \min_{X_i}~\Omega(X_0^{k+1},X_1^{k+1},\dots,X_{i-1}^{k+1},X_i,X_{i+1}^{k},\dots,X_M^{k})+\frac{1}{2}\left\Vert X_i-X_i^{k}\right\Vert_{\mathcal{S}_i}^2
\end{equation*}
admits a closed-form solution (see \eqref{u1-chap3} and \eqref{u2-chap3} below) for all iterations $k\geq 0$ and blocks $i=0,\dots,M$.
\end{lemma}

\begin{proof}
For $i=1,\dots,M$, denote
\begin{equation*}
  \begin{aligned}
  &\Xi_i=V_2\otimes V_2F_i+V_2F_i\otimes V_2,\\
  &\mathcal{L}^\star(X_{-i})=\mathcal{L}^\star(X_0,\dots,X_M)-D_1^\top[(V_2\otimes V_2F_i)^\top+(V_2F_i\otimes V_2)^\top]T_2^\top X_i,
  \end{aligned}
\end{equation*}
i.e., $\mathcal{L}^\star(X_{-i})$ denotes the part of $\mathcal{L}^\star$ independent of $X_i$.
We obtain
\begin{equation*}
\begin{aligned}
  \partial_{X_i}\Omega=&-T_2(V_2\otimes V_2){\rm vec}(Q)-T_2\Xi_iD_1(D_1^\top D_1)^{-1}D_1^\top\xi_n+\frac{1}{\beta}T_2\Xi_iD_1(D_1^\top D_1)^{-1}\mathcal{L}^\star(X_{-i})\\
  &+\frac{1}{\beta}T_2\Xi_iD_1(D_1^\top D_1)^{-1}D_1^\top \Xi_i^\top T_2^\top X_i+\delta_{\widetilde{\Gamma}_+^n}(X_i).
\end{aligned}
\end{equation*}
Letting
\begin{equation*}
  \begin{aligned}
  &H_i(X_{-i})=T_2(V_2\otimes V_2){\rm vec}(Q)+T_2\Xi_iD_1(D_1^\top D_1)^{-1}D_1^\top\xi_n-\frac{1}{\beta}T_2\Xi_iD_1(D_1^\top D_1)^{-1}\mathcal{L}^\star(X_{-i}),\\
  &L_i=\frac{1}{\beta}T_2\Xi_iD_1(D_1^\top D_1)^{-1}D_1^\top \Xi_i^\top T_2^\top,
  \end{aligned}
\end{equation*}
we have
\begin{equation*}
   \partial_{X_i}\Omega=-H_i(X_{-i})+L_iX_i+\delta_{\widetilde{\Gamma}_+^n}(X_i).
\end{equation*}
The subproblem w.r.t. $X_i$ exhibits the following form:
\begin{align}\label{barX-chap4}
{X}_{i}^{k+1}=\underset{X_i}{\operatorname*{argmin}}~&\Omega(X_0^k,X_1^k,\dots,X_{i-1}^k,X_i,{X}_{i+1}^{k+1},\dots,{X}_M^{k+1})+\frac{1}{2}\Vert X_i-X_i^k\Vert_{\mathcal{S}_i}^2,
\end{align}
where $\mathcal{S}_i$ is a positive linear operator given by
$\mathcal{S}_i(X)=\rho_i\mathcal{I}(X)-L_iX$ with $\rho_i=\max\{{\rm eig}(L_i)\}$
and an identity operator $\mathcal{I}$. The Fermat optimality condition implies 
\begin{equation*}
  \begin{aligned}
  0&\in-H(X^{k}_{-i})+L_iX_i+\delta_{\widetilde{\Gamma}_+^n}(X_i)+\mathcal{S}_i(X_i-X_i^k)\\
  &=\rho_i X_i+\delta_{\widetilde{\Gamma}_+^n}(X_i)-H(X^k_{-i})-\rho_i X_i^k+L_i X_i^k
  \end{aligned}
\end{equation*}
with $X^{k}_{-i}=(X_0^{k},\dots,X_{i-1}^k,X_{i+1}^{k+1},\dots,X_{M}^{k+1})$. Denoting $\Delta_i^k=\rho_i^{-1}(H_i(X_{-i}^k)+\rho_iX_i^k-L_iX_i^k)$, we have
\begin{equation}\label{u1-chap3}
  {X}_{i}^{k+1}=(I+\rho_i^{-1}\partial\delta_{\widetilde{\Gamma}_+^n})^{-1}(\Delta_i^k)=\Pi_{\widetilde{\Gamma}_+^n}(\Delta_i^k).
\end{equation}

For $i=0$, we have
\begin{equation*}
  \partial_{X_0}\Omega=-H_0(X_{-0})+L_0X_0+\delta_{\widetilde{\Gamma}_+^p}(X_0),
\end{equation*}
where
\begin{equation*}
  \begin{aligned}
  &\mathcal{L}^\star(X_{-0})=\mathcal{L}^\star(X_0,\dots,X_M)+X_0,\\
  &L_0=\frac{1}{\beta}(D_1^\top D_1)^{-1},\\
  &H_0(X_{-0})=-(D_1^\top D_1)^{-1}D_1^\top\xi_n+\frac{1}{\beta}(D_1^\top D_1)^{-1}\mathcal{L}^\star(X_{-0}).
  \end{aligned}
\end{equation*}
Similarly, the solution $X_{0}^{k+1}$ of subproblem w.r.t. $X_0$ is given by
\begin{equation*}
\underset{X_0}{\operatorname*{argmin}}~\Omega(X_0,{X}_1^{k+1},\dots,{X}_M^{k+1})+\frac{1}{2}\Vert X_0-X_0^k\Vert_{\mathcal{S}_0}^2,
\end{equation*}
where the positive linear operator $\mathcal{S}_0$ is given by
$\mathcal{S}_0(X)=\rho_0\mathcal{I}(X)-L_0X$ with $\rho_0=\max\{{\rm eig}(L_0)\}$.
Hence, by the same token, we have
\begin{equation}\label{u2-chap3}
X_{0}^{k+1}=(I+\rho_0^{-1}\partial\delta_{\widetilde{\Gamma}_+^p})^{-1}(\Delta_0^k)=\Pi_{\widetilde{\Gamma}_+^p}(\Delta_0^k)
\end{equation}
with $\Delta_0^k=\rho_0^{-1}(H_0(X^{k}_{-0})+\rho_0X_0^k-L_0X_0^k)$ and $X^{k}_{-0}=(X_1^{k+1},\dots,X_M^{k+1})$. Together with \eqref{u1-chap3}, \eqref{u2-chap3} and the fact that projection $\Pi_{\mathbb{S}_+^n}(\cdot)$ can be explicitly expressed by the eigenvalue decomposition (see Section 8.1.1 of \cite{boyd2004convex}), we finish the proof here.
\end{proof}

\begin{proposition}\label{thm_innerconverge}
Let $\{X^k\}_{k\geq 0}=\{(X_0^k,\dots,X_M^k)\}_{k\geq 0}$ be the sequence generated by \eqref{u1-chap3}, \eqref{u2-chap3} for solving optimization problem \eqref{dual2}. Then, it holds
\begin{equation*}
  {\Theta}(X^k)-{\Theta}^*\leq\mathcal{O}\left(\frac{1}{k}\right),~\forall k\geq 1.
\end{equation*}
\end{proposition}

\begin{proof}
Denote
\begin{equation*}
\begin{aligned}
&g(X_0,\dots,X_M)=\frac{1}{2\beta}\Vert (D_1^\top D_1)^{-\frac{1}{2}}\mathcal{L}^\star(X_0,X_1,\dots,X_M)\Vert^2,\\
&h_0(X_0)=\langle X_0,(D_1^\top D_1)^{-1}D_1^\top\xi_n\rangle,\\
&h_i(X_i)=-\langle X_i,T_2(V_2^\top\otimes V_2){\rm vec}(Q)+T_2[(V_2\otimes V_2F_i)+(V_2F_i\otimes V_2)]D_1(D_1^\top D_1)^{-1}D_1^\top\xi_n\rangle,
\end{aligned}
\end{equation*}
for $i=1,\dots,M$. Then, optimization problem \eqref{dual2} is equivalent to the following standard form:
\begin{equation*}
\begin{aligned}
\min\limits_{X_0,\dots,X_M} &\widehat{\Omega}(X_0,\dots,X_M)\doteq g(X_0,\dots,X_M)+\sum_{i=0}^{M}h_i(X_i)\\
{\rm s.t.} ~~~&X_0\times X_1\dots\times X_M\in\widetilde{\Gamma}_+^p\times\widetilde{\Gamma}_+^n\times\dots\times\widetilde{\Gamma}_+^n.
\end{aligned}
\end{equation*}
For $i=0,\dots,M$, let
\begin{equation*}
u_i(X_i,X^k)=g(X_0^k,\dots,X_i,\dots,X_M^k)+\frac{1}{2}\Vert X_i-X_i^k\Vert_{\mathcal{S}_i}^2
\end{equation*}
with $X^k=(X_0^k,\dots,X_M^k)$. Then, by the update rule \eqref{u1-chap3}, \eqref{u2-chap3}, pBCD can be regarded as a special case of BSUM, while $\{u_i\}_{i=0}^M$ is a sequence of approximation functions of $\widehat{\Omega}$; see Assumption B of \cite{b6}. Hence, this proposition is a direct corollary of Theorem 1 of \cite{b6}.
\end{proof}

For $i=1,\dots, M$, the relative residual error can be defined as
\begin{equation*}
\begin{aligned}
{\rm err}_{X_0}^k&=\frac{\Vert X_0^k-\Pi_{\widetilde{\Gamma}_+^p}(X_0^k-L_0X_0^k+H_0)\Vert}{1+\Vert X_0^k\Vert+\Vert L_0X_0^k-H_0 \Vert},\\
{\rm err}_{X_i}^k&=\frac{\Vert X_i^k-\Pi_{\widetilde{\Gamma}_+^n}(X_i^k-L_iX_i^k+H_i)\Vert}{1+\Vert X_i^k\Vert+\Vert L_iX_i^k-H_i \Vert}.
\end{aligned}
\end{equation*}
Let
\begin{equation}\label{u5}
{\rm err}^k=\max\{{\rm err}_{X_0}^k,{\rm err}_{X_1}^k,\dots,{\rm err}_{X_M}^k\},
\end{equation}
and the inner layer optimization process will terminate if ${\rm err}^k<\epsilon$.
By \eqref{sol-svecW},
\begin{equation}\label{u6-chap3}
  \mathrm{svec}(W)=\beta^{-1}(D_1^\top D_1)^{-1}(\beta D_1^\top \xi_n-\mathcal{L}^\star(X_0^k,X_1^k,\dots,X_M^k))
\end{equation}
is the optimal solution of primal problem \eqref{zn-1}. Based on the aforementioned discussion in this subsection, the solver of subproblem w.r.t. $z_k$  is summarized as Algorithm \ref{pBCD-code}.

\begin{breakablealgorithm}
    \caption{pBCD}
    \begin{algorithmic}
    \State \textbf{Given} : Initial point $(X_0^0,X_1^0,\dots,X_M^0)$, stopping criterion parameter $\epsilon$
    \State \textbf{Result} : ${\rm vec}(W)$
    \For{$k=0,1,2,...$}

        \State Update ${X}_M^{k+1},\dots,{X}_1^{k+1},X_0^{k+1}$ by \eqref{u1-chap3} and \eqref{u2-chap3}
        \State Determine ${\rm err}^{k+1}$ by (\ref{u5})
        \If{${\rm err}^{k+1}<\epsilon$}
           \State Compute ${\rm svec}(W)$ by (\ref{u6-chap3})\\
            ~~~~~~~~~~\Return ${\rm vec}(W)=D_1{\rm svec}(W)$
        \EndIf
    \EndFor
    \end{algorithmic}
    \label{pBCD-code}
\end{breakablealgorithm}

\begin{remark}\label{explain_of_solver}
 In this subsection, we propose a single-timescale first-order optimization algorithm to solve subproblem \eqref{zn-1}. However, various algorithms exist that can also address this subproblem. Compared to Algorithm \ref{pBCD-code}, these existing methods have the following limitations.

 \begin{itemize}
 	\item We first consider interior-point methods (IPMs), which exhibit three significant drawbacks. (1) Simone \textit{et al.} (2022) \cite{de2022sparse} point out every IPMs iteration should solve the associated linear systems with worse condition number, rendering numerical instability phenomenon; (2)  Lu and Yang (2023) \cite{lu2023practical} claim that IPMs are hard to take advantage of parallel computing massively due to the sequential nature of factorization; (3) Lu and Yang also highlight that IPMs may often raise an out-of-memory error and cannot take advantage of modern computing architectures, such as GPUs and distributed computing.
 	\item We then consider classic BCD-type methods, two-timescale algorithms that suffer from high computational complexity. Specifically, every classic BCD iteration necessitates solving the following problem: 
 	 \begin{equation*}
 	\min_{X_i}~\Omega(X_0^{k+1},X_1^{k+1},\dots,X_{i-1}^{k+1},X_i,X_{i+1}^{k},\dots,X_M^{k}),~~i=0,1,\dots, M,
 \end{equation*}
 which has no closed-form solution. Hence, we need to utilize a certain iterative algorithm to solve the above problem, making classic BCD a two-timescale method. 
 \end{itemize}
In summary, this subsection only demonstrates the feasibility of solving subproblem \eqref{zn-1} using a single-timescale first-order optimization algorithm. However, when the problem is small or sufficient computational resources are available, IPMs or classical BCD methods can be employed. 
\end{remark}

\subsection{Solver of subproblem w.r.t. $\widetilde{P}_n$}\label{subsubpn}

In this subsection, we will design the solver for subproblem w.r.t. $\widetilde{P}_n$:
\begin{equation}\label{subproblem-p}
\widetilde{P}_{n+1}\in \mathbf{prox}_{\mu^{-1}g}(\widetilde{P}_n-\mu^{-1}\rho\mathcal{B}^\top(\mathcal{A}\widetilde{W}_n+\mathcal{B}\widetilde{P}_n)).
\end{equation}
First, we introduce some important notations.
The matrix $U_\mathbb{I}$ denotes the submatrix of the $n$-dimensional identity matrix $I_n$ constructed from the columns corresponding to the index set $\mathbb{I}$, and in this notation, $x_\mathbb{I}=U_\mathbb{I}^\top x$.
The set $S_j(v)$ comprises all the index sets of size $j$ containing the $j$-largest elements in vector $v$.
Let's first define a way to index blocks of matrix $P$. For a matrix $P\in\mathbb{R}^{m \times n}$ divided into $s \times t$ blocks ($P=\mathbf{block}_{s,t}(P_{11},P_{12},\dots,P_{st})$), we define the index set $\mathbb{I}_{ij}$ for each block $(i,j)$ as follows:
\begin{equation*}
  \text{For each block }(i,j) \text{ with } i \in [s], j \in [t], \text{ we define an index set }\mathbb{I}_{ij}.
\end{equation*}
These index sets have the following properties:
\begin{itemize}
\item Each $\mathbb{I}_{ij}$ is a subset of $[mn]$ that indexes the entries in block $(i,j)$.
\item $[\mathrm{vec}(P)]_{\mathbb{I}_{ij}}\doteq\widetilde{P}_{\mathbb{I}_{ij}}$ gives us the vectorized form of block $P_{ij}$.
\item The sets $\mathbb{I}_{ij}$ partition $[mn]$ (they are disjoint and their union is $[mn]$).
\end{itemize}
To work with these blocks sequentially, we convert the double index $(i,j)$ to a single index $\ell$ by sorting lexicographically:
\begin{equation*}
  \mathbb{I}^*_{\ell} = \mathbb{I}_{ij} \text{ for } \ell = i+s(j-
  1) \in [st].
\end{equation*}
Finally, we define an indicator function $\pi: \mathbb{R}^{mn} \to \mathbb{R}^{st}$ that tells which blocks are non-zero:
\begin{equation*}
  \pi(\widetilde{P})_{\ell}=\left\{
  \begin{aligned}
  &1,~~\text{if block }\ell\text{ is non-zero } (\widetilde{P}_{\mathbb{I}^*_{\ell}}\neq \mathbf{0}),\\
  &0,~~\text{if block }\ell\text{ is zero } (\widetilde{P}_{\mathbb{I}^*_{\ell}}= \mathbf{0}).
    \end{aligned}
  \right.
\end{equation*}

Using these definitions, we can now rewrite optimization problem \eqref{subproblem-p} as
\begin{equation}\label{subproblem2-p-chap3}
  \min_{\widetilde{P}}~\frac{\gamma}{\mu}\Vert \pi(\widetilde{P})\Vert_0+\frac{1}{2}\Vert\widetilde{P}-\varpi_n\Vert^2
\end{equation}
with $\varpi_n=\widetilde{P} _n-\mu^{-1}\rho\mathcal{B}^\top(\mathcal{A}\widetilde{W}_n+\mathcal{B}\widetilde{P}_n)$.
Let the constraint set $C_S=\{\widetilde{P}\in\mathbb{R}^{mn}\colon \Vert \pi(\widetilde{P})\Vert_0\leq S\}$, and in this subsection, we consider a more general problem:
\begin{equation}\label{subproblem3-p-chap3}
  \min_{\widetilde{P}\in C_S}~\frac{\gamma}{\mu}\Vert \pi(\widetilde{P})\Vert_0+\frac{1}{2}\Vert\widetilde{P}-\varpi_n\Vert^2
\end{equation}
with predefined $S$. Obviously, problem \eqref{subproblem3-p-chap3} will degenerate to problem \eqref{subproblem-p} by setting $S=st$. Let operator $\mathcal{E}:2^{[st]}\to2^{[st]}$ be $\mathcal{E}(T)=\bigcup_{\ell\in T} \mathbb{I}^*_\ell$ and function $w:\mathbb{R}^{mn}\to\mathbb{R}^{st}$ be $w(\varpi_n)_\ell=\Vert ({\varpi_n})_{\mathbb{I}^*_\ell}\Vert^2$.

In practical algorithm implementations \eqref{PAML}, we typically only need to extract a single arbitrary element from the solution set of problem \eqref{subproblem3-p-chap3}, rather than being concerned with its complete structural characterization.
According to Theorem 3.2 of \cite{beck2019optimization}, the following algorithm can be used to obtain such a solution:

\begin{breakablealgorithm}
    \caption{Group-$\ell_0$-norm Proximal Operator Solver}
    \begin{algorithmic}[1]
    \State \textbf{Input} : sparse constraint $S$
    \State \textbf{Output} : $\widetilde{P}_{n+1}$
    \State Compute $\varpi_n=\widetilde{P}_n-\mu^ {-1}\rho\mathcal{B}^\top(\mathcal{A}\widetilde{W}_n+\mathcal{B}\widetilde{P}_n)$
    \State Arbitrarily select $\mathbb{I}^*\in S_{S}(w(\varpi_n))$
    \State Set $R=\left\{\ell\in\mathbb{I}^*\colon w(\varpi_n)_\ell>\frac{2\gamma}{\mu}\right\}$
    \State \Return $\widetilde{P}_{n+1}=U_{\mathcal{E}(R)}(\varpi_n)_{\mathcal{E}(R)}$
    \end{algorithmic}
    \label{prox-solver}
\end{breakablealgorithm}

\begin{remark}\label{constraint-ell0}
When solving optimization problem \eqref{Pen_1} using algorithm \eqref{PAML}, no sparsity constraint set $C_{S}$ is introduced. Thus, when applying algorithm \ref{prox-solver} to solve subproblem \eqref{subproblem-p} for $\widetilde{P}_{n+1}$, we simply set ${S} = st$. Notably, algorithm \eqref{PAML} can be extended to solve the following optimization problem:
\begin{equation}\label{Pen_2}
  \min_{\widetilde{W},\widetilde{P}}~~ F(\widetilde{W},\widetilde{P})= f(\widetilde{W})+\gamma \Vert \pi(\widetilde{P})\Vert_0+\delta_{C_{S}}(\widetilde{P})+H(\widetilde{W},\widetilde{P}).
\end{equation}
This extension requires only a minor modification to the iterative scheme \eqref{PAML}, namely defining
$$g(\widetilde{P})\doteq \gamma \Vert \pi(\widetilde{P})\Vert_0+\delta_{C_{S}}(\widetilde{P}).$$
Since $C_{S}$ is a closed set, $g$ maintains the property of being a closed function. Optimization problem \eqref{Pen_2} has particular relevance in distributed control systems, where it naturally captures the practical constraint that a multi-agent system's communication topology should be limited to at most $S$ communication links. Such constraints frequently arise in industrial applications due to hardware costs, bandwidth limitations, and other practical considerations.
\end{remark}

\section{Convergence Analysis}

In this section, we provide a  convergence analysis for algorithm \eqref{PAML} when applied to optimization problem \eqref{Pen_1}. While convergence analysis exists for similar algorithms, see Bo\c{t} \textit{et al.} (2019) \cite{bot2019proximal} that study problems of the form:
\begin{equation}\label{CA1}
  \min_{x,y}~~A(x)+B(y)+C(x,y).
\end{equation}
However, the results of \cite{bot2019proximal} rely on some critical assumptions that do not hold in our settings. Specifically, Theorem 2.10 of \cite{bot2019proximal} establishes convergence to a Karush-Kuhn-Tucker (KKT) point only when either:
\begin{itemize}
  \item function $C(x,y)$ is coercive (a function $f$ is coercive if
$\|x\| \to \infty \Rightarrow f(x) \to +\infty$); or
  \item functions $A(x)$ and $B(y)$ are coercive.
\end{itemize}
A key distinguishing feature of our problem \eqref{Pen_1} is that neither the function $g$ nor $H$ satisfies the coercivity property required by \cite{bot2019proximal}. This fundamental difference means that we cannot apply existing convergence results directly. 

Obviously, we have the following fact:
\begin{equation*}
  \begin{aligned}
  &\nabla_{\widetilde{W}} H(\widetilde{W},\widetilde{P})=\rho \mathcal{A}^\top (\mathcal{A}\widetilde{W}+\mathcal{B}\widetilde{P}),\\
  &\nabla_{\widetilde{P}} H(\widetilde{W},\widetilde{P})=\rho \mathcal{B}^\top (\mathcal{A}\widetilde{W}+\mathcal{B}\widetilde{P}).
  \end{aligned}
\end{equation*}
Hence, for any fixed $\widetilde{P}$, the mapping $\nabla_{\widetilde{W}} H(\cdot,\widetilde{P})$ is $\kappa_1 \doteq \rho\Vert\mathcal{A}^\top \mathcal{A}\Vert$-Lipschitz continuous; for any fixed $\widetilde{W}$, the mapping $\nabla_{\widetilde{P}} H(\widetilde{W},\cdot)$ is $\kappa_2 \doteq \rho\Vert\mathcal{B}^\top \mathcal{B}\Vert$-Lipschitz continuous, while $\nabla_{\widetilde{W}} H(\widetilde{W},\cdot)$ is $\kappa_3 \doteq \rho\Vert\mathcal{A}^\top \mathcal{B}\Vert$-Lipschitz continuous.
Denote
\begin{equation*}
  \begin{aligned}
  &\mathbf{R}\doteq \mathbb{R}^{p^2}\times \mathbb{R}^{mn}\times \mathbb{R}^{p^2}\times \mathbb{R}^{p^2}\times\mathbb{R}^{p^2}\times \mathbb{R}^{p^2},\\
  &\mathbf{X}\doteq (\widetilde{W},\widetilde{P},z,u,\widetilde{W}^\prime,u^\prime),\\
  &\mathbf{X}_n\doteq (\widetilde{W}_n,\widetilde{P}_n,z_n,u_n,\widetilde{W}_{n-1},u_{n-1})~~\forall n\geq1,
    \end{aligned}
\end{equation*}
and define regularized augmented Lagrangian $\Psi:\mathbf{R}\to \mathbb{R}\cup \{+\infty\}$:
\begin{equation}\label{Psi-Lagrangian}
  \begin{aligned}
  \Psi(\mathbf{X})=&f(z)+g(\widetilde{P})+H(\widetilde{W},\widetilde{P})+\langle u,\widetilde{W}-z\rangle+\frac{\beta}{2}\Vert \widetilde{W}-z\Vert^2\\
  &+C_0\Vert(u-u^\prime)+\sigma(\tau-\beta)(\widetilde{W}-\widetilde{W}^\prime)\Vert^2+C_1\Vert\widetilde{W}-\widetilde{W}^\prime\Vert^2
  \end{aligned}
\end{equation}
with
\begin{equation*}
  C_0=\frac{4(1-\sigma)}{\sigma^2\beta},~~C_1=\frac{8(\sigma\tau+\kappa_1)^2}{\sigma\beta}.
\end{equation*}
The definitions of Kurdyka-{\L}ojasieicz (KL) property and semi-algebraic property are standard; see Definition 3 and Definition 5 of \cite{bolte2014proximal}, respectively.

\begin{remark}
In this remark, we aim to clarify the distinctions between this paper's convergence analysis methodologies and those in the existing literature.  For this purpose, the convergence proof is organized into the following steps, with each step accompanied by a discussion on its relation to and deviation from previous works.
\begin{itemize}
\item \textbf{Step 1: Sufficient Decrease Property of Lagrangian.} In other words, we aim to show
	\begin{equation*}
		\Psi_n -\Psi_{n+1}\geq C\Vert(\widetilde{W}_{n+1},\widetilde{P}_{n+1},z_{n+1},u_{n+1})-(\widetilde{W}_n,\widetilde{P}_n,z_n,u_n)\Vert^2
	\end{equation*}
	with constant $C$. Such decrease property will be established in Lemma \ref{desecent_lemma} below, while its proof follows exactly the same line of Lemma 2.5 of \cite{bot2019proximal}.
	
\item \textbf{Step 2: Boundness of Iterative Sequence $\{(\widetilde{W}_n,\widetilde{P}_n,z_n,u_n)\}_{n\geq 0}$.} In general, the boundness of the PALM iterative sequence $\{(\widetilde{W}_n,\widetilde{P}_n,z_n,u_n)\}_{n\geq 0}$ is hard to verify. Bolte \textit{et al.} simply assume the boundedness condition holds without further verification (see Lemma 5 of \cite{bolte2014proximal}); Bo\c{t} \textit{et al.} (2019)  ensure the boundedness through imposing coercivity condition as mentioned above (see Lemma 2.5 of \cite{bot2019proximal}); while Li and Pong (2015) prove the boundness through the analogous coercivity condition (see Theorem 3 of \cite{li2015global}). In contrast, Theorem \ref{boundness-thm} below will establish the boundedness without invoking any coercivity condition, depending on the special structure of the LQ problem.
\item \textbf{Step 3: Proof of KL Property.} The existing works assume the objective function exhibits the KL property; see, for example, Section 3.5 of \cite{bolte2014proximal}. However, we provide a rigorous proof of the KL property of $\Psi$; see Theorem \ref{semialgethm} below.
\item \textbf{Step 4: Proof of Convergence to Stationary Point.} In fact, provided sufficient decrease property, boundness of iterative sequence, and KL property, the convergence proof (including but not limited to PALM, ADMM, and their variants)  has become a standard and routine argument, which can be found in many existing works or even textbooks on nonconvex optimization. We claim the convergence of the sequence in Theorem \ref{convergence-stationary} by citing Theorem 3.4 of \cite{bot2019proximal}, while Theorem 1 of \cite{bolte2014proximal}, Theorem 4.4 of \cite{themelis2020douglas} and Theorem 4 of \cite{li2015global} all do the same thing. 
\item \textbf{Step 5: Calculating {\L}ojasieicz Exponent and Convergence Rate Analysis.} Our analysis on this part is relatively brief, while Remark \ref{calculateKL} provides insights for calculating {\L}ojasieicz exponent. Moreover, given {\L}ojasieicz exponent and convergence guarantee, the convergence rate analysis is again a standard and routine arguments. By directly quoting Theorem 10 of \cite{bot2019proximal}, we establish Theorem \ref{convergencerateofpaml} to illustrate the convergence rate of PALM.

\end{itemize}

\end{remark}

\begin{lemma}\label{desecent_lemma}
Let $\Psi_n = \Psi(\mathbf{X}_n)$ and suppose $2\tau \geq \beta$. Then, for any $n \geq 1$, the iterative sequence $\{(\widetilde{W}_n,\widetilde{P}_n,z_n,u_n)\}_{n\geq 0}$ generated by algorithm \eqref{PAML} satisfies:
\begin{equation}\label{descent-chap3}
  \Psi_{n+1}+C_2\Vert\widetilde{W}_{n+1}-\widetilde{W}_n\Vert^2+C_3\Vert\widetilde{P}_{n+1}-\widetilde{P}_n\Vert^2+C_4\Vert u_{n+1}-u_n\Vert^2\leq \Psi_n,
\end{equation}
where
\begin{equation*}
  \begin{aligned}
  &C_2=\tau-\frac{\kappa_1+\beta}{2}-\frac{4\sigma\tau^2}{\beta}-\frac{8(\sigma\tau+\kappa_1)^2}{\sigma\beta},\\
  &C_3=\frac{\mu-\kappa_2}{2}-\frac{8\kappa_3^2}{\sigma\beta},\\
  &C_4=\frac{1}{\sigma\beta}.
  \end{aligned}
\end{equation*}
\end{lemma}

\begin{proof}
This lemma is a direct corollary of Lemma 2.5 of \cite{bot2019proximal}.
\end{proof}

\begin{lemma}\label{selectrule}
Let
\begin{equation}\label{chosepara}
  \begin{aligned}
  &0<\sigma<\frac{1}{24},\\
  &\beta>\frac{4\kappa_1}{1-24\sigma}\left(4+3\sigma+\sqrt{24-168\sigma+9\sigma^2}\right)>0,\\
  &\mu>\kappa_2+\frac{16\kappa_3^2}{\sigma\beta}>0,\\
  &\max\left\{\frac{\beta}{2},\frac{\beta}{24\sigma}\left(1-\frac{16\kappa_1}{\beta}-\sqrt{\varsigma}\right)\right\}<\tau<\frac{\beta}{24\sigma}\left(1-\frac{16\kappa_1}{\beta}+\sqrt{\varsigma}\right)
  \end{aligned}
\end{equation}
with
\begin{equation*}
  \varsigma=1-\frac{32\kappa_1}{\beta}-\frac{128\kappa_1^2}{\beta^2}-\frac{24\kappa_1\sigma}{\beta}-24\sigma>0.
\end{equation*}
Then, it holds $\min\{C_2,C_3,C_4\}>0$, and there exists $\gamma_1,\gamma_2\neq 0$ such that
\begin{equation}\label{gamma-chap3}
  \begin{aligned}
  &\frac{1}{\gamma_1}-\frac{\kappa_1}{2\gamma_1^2}=\frac{1}{\beta},~~\frac{1}{\gamma_2}-\frac{\kappa_1}{2\gamma_2^2}=\frac{2}{\beta}.
  \end{aligned}
\end{equation}
\end{lemma}

\begin{proof}
This lemma is a direct corollary of Lemma 2.6 of \cite{bot2019proximal}.
\end{proof}

\begin{theorem}\label{boundness-thm}
  Assume that parameters $\mu, \beta, \tau, \sigma$ satisfy \eqref{chosepara}; then the iterative sequence $\{(\widetilde{W}_n, \widetilde{P}_n, z_n, u_n)\}_{n \geq 0}$ generated by algorithm \eqref{PAML} is bounded.
\end{theorem}

\begin{proof}
See Appendix for the proof.
\end{proof}

\begin{theorem}\label{semialgethm}
The regularized augmented Lagrangian $\Psi$ \eqref{Psi-Lagrangian} is semi-algebraic.
\end{theorem}

\begin{proof}
First, we prove that $f(z)$ is a semi-algebraic function, where $f(z)$ takes the following form:
\begin{equation*}
  \begin{aligned}
  f(z)&=\overbrace{\langle {\rm vec}(R),z\rangle}^{\mathrm{\uppercase\expandafter{\romannumeral1}} }+\overbrace{\delta_{\Gamma_+^p}(z)}^{\mathrm{\uppercase\expandafter{\romannumeral2}}}+\overbrace{\delta_{\Gamma_+^n}(\mathcal{H}_1(z))+\cdots+\delta_{\Gamma_+^n}(\mathcal{H}_M(z))}^{\mathrm{\uppercase\expandafter{\romannumeral3}}}
  \end{aligned}
\end{equation*}
with $\{\mathcal{H}_i\}_{i=1}^M$ a sequence of linear operator.
The term \uppercase\expandafter{\romannumeral1} is a linear function that is clearly semi-algebraic, while the term \uppercase\expandafter{\romannumeral2} is the indicator function of positive semidefinite cone $\Gamma_+^p$. According to Example 2 of \cite{bolte2014proximal}, this term is semi-algebraic. We now prove that the term \uppercase\expandafter{\romannumeral3} is also semi-algebraic. Due to the symmetry, it suffices to show that $\delta_{\Gamma_+^n}(\mathcal{H}_1(z))$ is a semi-algebraic function. Denote
\begin{equation*}
\begin{aligned}
  &S_1=\{(z,p,t)\in\mathbb{R}^{p^2}\times \mathbb{R}^{n^2}\times \mathbb{R}\colon p=\mathcal{H}_1(z)\},\\
  &S_2=\{(z,p,t)\in\mathbb{R}^{p^2}\times \mathbb{R}^{n^2}\times \mathbb{R}\colon t=\delta_ {\Gamma_+^n}(p)\},
\end{aligned}
\end{equation*}
and similarly, we know that $S_1$ and $S_2$ are semi-algebraic sets, and thus their intersection $S_1 \cap S_2$ is also semi-algebraic. By the Tarski-Seidenberg theorem (see Theorem 1.9 of \cite{coste2000introduction}), we conclude that
\begin{equation*}
  S=\{(z,t)\in\mathbb{R}^{p^2}\times \mathbb{R}\colon t=\delta_ {\Gamma_+^n}(\mathcal{H}_1(z))\}
\end{equation*}
is a semi-algebraic set. This implies that the function $\delta_{\Gamma_+^n}(\mathcal{H}_1(z))$ is semi-algebraic, and hence $f$ is a semi-algebraic function.

We now prove that $g(\widetilde{P})$ is a semi-algebraic function. For any index set $\mathbb{I}^* \subseteq [st]$, define
\begin{equation*}
  J_l^{\mathbb{I}^*} =
  \left\{
  \begin{array}{ll}
    \{\mathbf{0}_{k_l}\}, & \text{if } l \in \mathbb{I}^*, \\
    \mathbb{R}^{k_l} \setminus \{\mathbf{0}_{k_l}\}, & \text{if } l \notin \mathbb{I}^*,
  \end{array}
  \right.
\end{equation*}
where $k_l = m_{(l \bmod s)} n_{\left(\lfloor \frac{l}{s} \rfloor + 1 \right)}$ with $m_0 \doteq m_s$.
From $g(\widetilde{P})=\gamma \Vert \pi(\widetilde{P})\Vert_0$, we have
\begin{equation*}
  \sigma_*(\mathbf{graph}(g))=\bigcup_{\mathbb{I}^*\subseteq [st]}\left(\prod_{l=1}^{st}J_l^{\mathbb{I}^*}\right)\times \{\gamma(st-\mathrm{card}(\mathbb{I}^*))\},
\end{equation*}
where $\sigma_*(\mathbf{graph}(g))$ denotes the image of the graph of $g$ under a certain coordinate permutation (however, we do not care about the specific form of the permutation $\sigma_*$), this transformation preserves the algebraic properties of the set. Therefore, $\sigma_*(\mathbf{graph}(g))$ is a semi-algebraic set, which implies that $\mathbf{graph}(g)$ is semi-algebraic. Consequently, $g$ is a semi-algebraic function.

The remaining terms in $\Psi$ are all quadratic and are semi-algebraic by Section 2.2.1 of \cite{coste2000introduction}. Since the class of semi-algebraic functions is closed under addition,  $\Psi$ is a semi-algebraic function. It completes the proof of the theorem. 
\end{proof}

\begin{corollary}\label{KL-psi}
The regularized augmented Lagrangian functional $\Psi$ \eqref{Psi-Lagrangian} is a KL function.
\end{corollary}

\begin{proof}
By Theorem \ref{semialgethm}, $\Psi$ is a proper, closed, and semi-algebraic function, and it follows from the Bolte-Daniilidis-Lewis theorem \cite{bolte2007lojasiewicz} that $\Psi$ satisfies the KL property.
\end{proof}

Based on the above discussion, we summarize the iterative scheme \eqref{PAML} in the form of the following pseudocode.

\begin{breakablealgorithm}
    \caption{Proximal Alternating Linearized Minimization (PALM)}
    \begin{algorithmic}[1]
    \State \textbf{Input} : parameters $\mu,\beta,\tau,\sigma$ satisfy \eqref{chosepara}, initial point $(\widetilde{W}_0, \widetilde{P}_0, z_0, u_0)$
    \State \textbf{Output} : $\widetilde{W}_{n+1}$
    \For{$k=0,1,2,...$}
        \State Update $\widetilde{P}_{n+1}\leftarrow$Algorithm \ref{prox-solver}
        \State Update $z_{n+1}\leftarrow$Algorithm \ref{pBCD-code}
        \State Update $\widetilde{W}_{n+1}=\widetilde{W}_n-\tau^{-1}(\nabla_{\widetilde{W}}H(\widetilde{W}_n,\widetilde{P}_{n+1})+u_n+\beta(\widetilde{W}_n-z_{n+1}))$
        \State Update $u_{n+1}=u_n+\sigma \beta (\widetilde{W}_{n+1}-z_{n+1})$
        \If{Stopping Criterion$==\mathbf{True}$}
            \Return $\widetilde{W}_{n+1}$
        \EndIf
    \EndFor
    \end{algorithmic}
    \label{PALM-code}
\end{breakablealgorithm}

\begin{theorem}\label{convergence-stationary}
Let $\{(\widetilde{W}_n, \widetilde{P}_n, z_n, u_n)\}_{n \geq 0}$ be the sequence generated by Algorithm \ref{PALM-code}, and suppose that the parameter selection criterion \eqref{chosepara} is satisfied. Then, the following properties hold:
\begin{itemize}
  \item the sequence $\{(\widetilde{W} _n,\widetilde{P}_n,z_n,u_n)\}_{n\geq 0}$ has finite length, namely,
        \begin{equation*}
    \begin{aligned}
    &\sum_{n\geq 0}\Vert\widetilde{W}_{n+1}-\widetilde{W}_n\Vert<+\infty,~~\sum_{n\geq 0}\Vert\widetilde{P}_{n+1}-\widetilde{P}_n\Vert<+\infty,\\
    &\sum_{n\geq 0}\Vert z_{n+1}-z_n\Vert<+\infty,~~\sum_{n\geq 0}\Vert u_{n+1}-u_n\Vert<+\infty;
    \end{aligned}
  \end{equation*}
  \item the sequence $\{(\widetilde{W} _n,\widetilde{P}_n,z_n,u_n)\}_{n\geq 0}$ converges to a KKT point of optimization problem \eqref{Pen_1}, i.e.,
        \begin{equation*}
    (\widetilde{W} _n,\widetilde{P}_n,z_n,u_n)\to(\widetilde{W}^*,\widetilde{P}^*,z^*,u^*),
  \end{equation*}
  where $(\widetilde{W}^*,\widetilde{P}^*,z^*,u^*)$ satisfies
    \begin{equation*}
    \left\{
    \begin{aligned}
    &0=u^*+\rho \mathcal{A}^\top (\mathcal{A}\widetilde{W}^*+\mathcal{B}\widetilde{P}^*),\\
    &0\in\partial g(\widetilde{P}^*)+\rho \mathcal{B}^\top (\mathcal{A}\widetilde{W}^*+\mathcal{B}\widetilde{P}^*),\\
    &u^*\in\partial f(z^*),\\
    &z^*=\widetilde{W}^*.
    \end{aligned}
    \right.
  \end{equation*}
\end{itemize}
\end{theorem}

\begin{proof}
It is easy to verify that $f, g, H \geq 0$, and that $H$ is a $\mathcal{C}^1$ function. According to Theorem \ref{boundness-thm}, the sequence $\{(\widetilde{W}_n, \widetilde{P}_n, z_n, u_n)\}_{n \geq 0}$ is bounded. Moreover, Lemma \ref{selectrule} ensures that $\min{C_2, C_3, C_4} > 0$. In combination with Corollary \ref{KL-psi}, the desired result follows directly from Theorem 3.4 of \cite{bot2019proximal}.
\end{proof}

We now proceed to analyze the convergence rate of Algorithm \ref{PALM-code}. To this end, we first introduce the {\L}ojasiewicz property.

\begin{definition}[\!\!\cite{bolte2007lojasiewicz}]
Let $\psi\colon\mathbb{R}^n\to\mathbb{R}\cup\{+\infty\}$ be proper and lower semicontinuous. Then, $\Psi$ satisfies the {\L}ojasiewicz property if for any critical point $\hat{v}$ of $\Psi$ there exists $c>0$, $\theta\in[0,1)$ and $\epsilon>0$ such that
\begin{equation*}
  |\psi(v)-\psi(\hat{v})|^\theta\leq c\cdot \mathrm{dist}(0,\partial \Psi(v))~~\forall v\in \mathrm{Ball}(\hat{v},\epsilon),
\end{equation*}
where $\mathrm{Ball}(\hat{v},\epsilon)$ denotes the open ball with center $\hat{v}$ and radius $\epsilon$. More precisely, $\theta$ is referred to as the {\L}ojasiewicz exponent of $\psi$.
\end{definition}

\begin{lemma}
The regularized augmented Lagrangian functional $\Psi$ \eqref{Psi-Lagrangian} satisfies the {\L}ojasiewicz property with the {\L}ojasiewicz exponent $\theta \in (0,1)$.
\end{lemma}

\begin{proof}
Since $\Psi$ is a semi-algebraic function, by the Puiseux lemma and Corollary 16 of \cite{bolte2007clarke}, the desingularizing function of $\Psi$ is given by $\varphi(s) = cs^{1-\theta}$, where $\theta \in (0,1)$. Therefore, $\Psi$ satisfies the {\L}ojasiewicz property with the {\L}ojasiewicz exponent $\theta \in (0,1)$.
\end{proof}

\begin{theorem}\label{convergencerateofpaml}
Let $\{(\widetilde{W}_n, \widetilde{P}_n, z_n, u_n)\}_{n \geq 0}$ be the sequence generated by Algorithm \ref{PALM-code}, and suppose that the parameter selection criterion \eqref{chosepara} is satisfied. Then, the following properties hold:
\begin{itemize}
  \item if $\theta \in (0, 1/2]$, then there exist constants $n_0\geq 0$, $P_{0,1},P_{0,2},P_{0,3},P_{0,4}>0$ and $P\in(0,1)$ such that for all $n\geq n_0$, we have
         \begin{equation*}
        \begin{aligned}
        &\Vert \widetilde{W}_{n}-\widetilde{W}^*\Vert\leq P_{0,1} P^n,~~\Vert \widetilde{P}_{n}-\widetilde{P}^*\Vert\leq P_{0,2} P^n,\\
        &\Vert {z}_{n}-{z}^*\Vert\leq P_{0,3} P^n,~~\Vert {u}_{n}-{u}^*\Vert\leq P_{0,4} P^n;\\
        \end{aligned}
      \end{equation*}
  \item  if $\theta\in (1/2,1)$, then there exist constants $n_1\geq 0$, $P_{1,1},P_{1,2},P_{1,3},P_{1,4}>0$  such that for all $n\geq n_1$, we have
            \begin{equation*}
        \begin{aligned}
        &\Vert \widetilde{W}_{n}-\widetilde{W}^*\Vert\leq P_{1,1} n^{-\frac{1-\theta}{2\theta-1}},~~\Vert \widetilde{P}_{n}-\widetilde{P}^*\Vert\leq P_{1,2} n^{-\frac{1-\theta}{2\theta-1}},\\
        &\Vert {z}_{n}-{z}^*\Vert\leq P_{1,3} n^{-\frac{1-\theta}{2\theta-1}},~~\Vert {u}_{n}-{u}^*\Vert\leq P_{1,4}n^{-\frac{1-\theta}{2\theta-1}}.\\
        \end{aligned}
      \end{equation*}
\end{itemize}
\end{theorem}

\begin{proof}
Since $\Psi$ satisfies the {\L}ojasiewicz property, this theorem is a direct corollary of Theorem 3.10 of \cite{bot2019proximal}.
\end{proof}

\begin{remark}\label{calculateKL}
Based on the above theorem, we observe that algorithm \eqref{PAML} exhibits two distinct convergence behaviors: linear convergence when $\theta \in (0, 1/2]$, and sublinear convergence when $\theta \in (1/2, 1)$. To precisely characterize the convergence rate, we must determine the specific value of $\theta$. While determining the {\L}ojasiewicz exponent of a general semi-algebraic function (which satisfies the {\L}ojasiewicz property by the Puiseux lemma) is typically challenging, our problem has a special structure. Under some additional assumptions, {\L}ojasiewicz exponent of $\Psi$ can be explicitly determined. Specifically, the regularized augmented Lagrangian functional $\Psi$ takes the following special form:
  \begin{align*}
  \Psi(\mathbf{X})&=\delta_{\Gamma_+^p}(z)+\sum_{i=1}^{M}\delta_{\mathcal{H}_i^{-1}(\Gamma_+^n)}(z)+\gamma \Vert \pi(\widetilde{P})\Vert_0+\mathbf{Quadratic}(\mathbf{X})\\
  &=\delta_{\mathfrak{D}}(z)+\gamma \Vert \pi(\widetilde{P})\Vert_0+\mathbf{Quadratic}(\mathbf{X})~~(\text{with}~\mathfrak{D}\doteq \Gamma_+^p\cap (\cap_{i=1}^M\mathcal{H}_i^{-1}(\Gamma_+^n)))\\
  &=\delta_{\mathbf{D}}(\mathbf{X})+\gamma \Vert \pi({\mathbf{Proj}_{\widetilde{P}}(\mathbf{X})})\Vert_0+\mathbf{Quadratic}(\mathbf{X})~~(\text{with}~\mathbf{D}\doteq \mathbb{R}^{p^2}\times \mathbb{R}^{mn}\times\mathfrak{D}\times \mathbb{R}^{p^2}\times\mathbb{R}^{p^2}\times \mathbb{R}^{p^2}).
  \end{align*}
Here, $\mathbf{Quadratic}(\mathbf{X})$ represents all quadratic and linear terms in $\mathbf{X}$. By applying techniques from \cite{wu2021kurdyka, pan2019kl, yu2019deducing, bi2019kl20}, it can be shown that under suitable conditions, the {\L}ojasiewicz exponent of $\Psi$ is $\frac{1}{2}$. This implies that Algorithm \ref{PALM-code} achieves linear convergence.
\end{remark}

\begin{remark}
In Remark \ref{constraint-ell0}, we proved that when algorithm \eqref{PAML} solves the optimization problem
\begin{equation}\label{Pen_2-new}
  \min_{\widetilde{W},\widetilde{P}}~~ F(\widetilde{W},\widetilde{P})= f(\widetilde{W})+\gamma \Vert \pi(\widetilde{P})\Vert_0+\delta_{C_{S}}(\widetilde{P})+H(\widetilde{W},\widetilde{P}). 
\end{equation}
The subproblems can still be efficiently solved using the algorithms designed in Sections \ref{subsubzn} and \ref{subsubpn}. Here, we show that algorithm \eqref{PAML} maintains the same convergence properties when solving problem \eqref{Pen_2-new}.
To establish this, we only need to prove that $\delta_{C_{S}}(\widetilde{P})$ is a semi-algebraic function, since the non-negativity of $\delta_{C_{S}}(\widetilde{P})$ does not affect the proof of Theorem \ref{boundness-thm}. The semi-algebraic nature of $\delta_{C_{S}}(\widetilde{P})$ can be demonstrated by examining its graph structure:
\begin{equation*}
\begin{aligned}
  \sigma_{*}^\prime(\mathbf{graph}(\delta_{C_{{S}}}))=&\left\{\bigcup_{\mathbb{I}^*\subseteq [st],\mathrm{card}(\mathbb{I}^*)\geq st-S}\left(\prod_{l=1}^{st}J_l^{\mathbb{I}^*}\right)\times \{1\}\right\} \bigcup \\
  &\left\{\bigcup_{\mathbb{I}^*\subseteq [st],\mathrm{card}(\mathbb{I}^*)< st-S}\left(\prod_{l=1}^{st}J_l^{\mathbb{I}^*}\right)\times \{0\}\right\},
\end{aligned}
\end{equation*}
where $\sigma_{*}^\prime(\mathbf{graph}(\delta_{C_{S}}))$ represents a certain coordinate-permuted version of $\mathbf{graph}(\delta_{C_{S}})$ (again, we do not care about the specific form of the permutation $\sigma_{*}^\prime$). This structure clearly demonstrates that $\delta_{C_{S}}(\widetilde{P})$ is indeed semi-algebraic.
\end{remark}

\section{Direct ADMM}\label{section5}

In the previous section, we transformed the constrained optimization problem \eqref{sparse_feedback_6} into the unconstrained problem \eqref{Pen_1} using the penalty function approach. Based on this transformation, we provided a convergence analysis of algorithm \eqref{PAML} under the parameter selection criterion \eqref{chosepara}. However, the penalty method suffers from difficulty selecting an appropriate penalty parameter $\rho$. Nonconvex optimization may introduce many undesirable stationary points, making the algorithm more likely to get trapped in suboptimal solutions.
Hence, we hope to investigate the feasibility of using  the ADMM algorithm
\begin{equation}\label{ADMM-alg}
\left\{
  \begin{aligned}
  &u_{n+1}=\lambda_n-\beta(1-\xi)(\mathcal{A}\widetilde{W}_n+\mathcal{B}\widetilde{P}_n),\\
  &\widetilde{W}_{n+1}\in\underset{\widetilde{W}}{\operatorname*{argmin}}~ \mathcal{L}_\beta(\widetilde{W},\widetilde{P}_n,u_{n+1}),\\
  &\lambda_{n+1}=u_{n+1}+\beta (\mathcal{A}\widetilde{W}_{n+1}+\mathcal{B}\widetilde{P}_n),\\
  &\widetilde{P}_{n+1}\in\underset{\widetilde{P}}{\operatorname*{argmin}}~\mathcal{L}_\beta(\widetilde{W}_{n+1},\widetilde{P},\lambda_{n+1})
  \end{aligned}
\right.
\end{equation}
to solve constrained problem \eqref{sparse_feedback_6} directly. 
It is widely acknowledged that solving constrained problem \eqref{sparse_feedback_6} with nonsmooth $f, g$ and nonconvex $g$ is fundamentally challenging; see \cite{barber2024convergence}. Due to space limitations, we defer the detailed discussion to Part II of our paper \cite{feng2025nonconvex-part2}.

\section{Numerical Examples}\label{section6}

In this section, we provide several numerical examples to demonstrate the numerical superiority of the proposed algorithm.
Remarkably, although the content of Section \ref{section5} is not discussed in detail in this paper, we still present numerical results of the ADMM algorithm (introduced in Part II) to facilitate the reader’s understanding of the relative strengths and weaknesses of the PALM algorithm.
First, we begin by solving the optimization problem using existing methodologies. Let $W_2^\top$ exhibit the following block structure:
\begin{equation*}
  W_2^\top=\mathbf{block}_{s,t}(W_{2,11},\dots,W_{2,st}),
\end{equation*}
and by selecting sufficiently large $M$, optimization problem \eqref{sparse_feedback_6} can be reformulated by the following MISDP (Mixed-Integer Semidefinite Program) problem:
\begin{equation}\label{misdp}
  \begin{aligned}
    \min_{W\in\mathbb{S}^p,z\in\{0,1\}^{s\times t}}~~&\langle R,W\rangle+\gamma \sum_{i\in [s],j\in[t]}z_{ij}\\
    {\rm s.t.}~~&W\in\mathbb{S}^p_+,\\
    &-V_2(F_iW+WF_i^\top+Q)V_2^\top\in\mathbb{S}_+^{n},~~\forall i=1,\dots,M,\\
    &-V_{j1}WV_{j2}=0,~~\forall j=1,\dots,N,\\
    &\Vert W_{2,ij}\Vert_{\infty}\leq Mz_{ij},~~\forall i\in [s],j\in[t].
  \end{aligned}
\end{equation}
Problem \eqref{misdp} admits significant difficulties:
\begin{itemize}
  \item The selection of an appropriate value for $M$ poses significant challenges. A value that is too small may inadvertently exclude feasible solutions, rendering the problem artificially infeasible. Conversely, choosing a value that is too large can introduce numerical instability and substantially degrade solver performance. 
  \item MISDP problems are intrinsically challenging to solve due to their combinatorial and nonconvex nature. Remarkably, most mainstream optimization solvers (including CVXPY, Gurobi, and HiGHS) lack native support for MISDP. Although MOSEK can solve specific subclasses of MISDP, it does not accommodate the type of problem \eqref{misdp}.
\end{itemize}

To address these issues, Coey \textit{et al.} (2020) \cite{coey2020outer} introduce a new open-source MISDP solver, Pajarito (see \url{github.com/JuliaOpt/Pajarito.jl}), which uses an external mixed-integer linear (MILP) solver to manage the search tree and an external continuous conic solver for subproblems. However, as we will demonstrate in the following example, even Pajarito struggles to solve problem \eqref{misdp} on a small scale.

\begin{example}\label{exp1}
Consider $x=[x_1,x_2,x_3]^\top$ and a linear system
\begin{equation*}
\begin{aligned}
&\dot{x}=Ax+B_2u+B_1w,\\
&z=Cx+Du,\\
&u=-Kx,
\end{aligned}
\end{equation*}
where
\begin{equation*}
  \begin{aligned}
  &A=\begin{bmatrix}
  0 & 1 & 0\\
  0 & 0 & 1\\
  0 & 0 & 0
  \end{bmatrix},B_1=I_3,B_2=\begin{bmatrix}
  0.9315 & 0.7939\\
  0.9722 & 0.1061\\
  0.5317 & 0.7750
  \end{bmatrix},C=\begin{bmatrix}
  1 & 0 & 0\\
  0 & 0 & 0\\
  0 & 0 & 0
  \end{bmatrix},D=\begin{bmatrix}
  0 & 0\\
  1 & 0\\
  0 & 1
  \end{bmatrix},
  \end{aligned}
\end{equation*}
and the noise $w$ is characterized by a impulse disturbance vector.

\end{example}

\begin{solution}
We consider feedback gain $K$ with the following block structure:
\begin{equation}\label{block-structure}
	K = \mathbf{block}_{2,2}(K_{11}, K_{12}, K_{21}, K_{22}),
\end{equation}
where $K_{11}\in\mathbb{R}^{1\times 2}, K_{12}\in\mathbb{R}, K_{21}\in\mathbb{R}^{1\times 2} ,K_{22}\in\mathbb{R}$. Firstly, we formulate such an SF-LQ problem into MISDP problem \eqref{misdp} with $M=10^4$. We then solve this problem using the Pajarito solver in Julia v1.11. Notably, the Pajarito solver performs poorly on this task: it tends to produce feedback gain matrices that are only marginally stable, or fails to find a feasible solution after $10^4$ search iterations. To demonstrate this fact, we introduce additional constraint $W_1 - \delta I \succeq 0$ into problem \eqref{misdp}, where $\delta$ varies continuously in $[0.01,0.1]$; see Remark \ref{explain_thm1} for more insights. Figure \ref{eigenvalue_trajactory} below illustrates the trajectory of the eigenvalues of $A-B_2{W_{2,\delta}}^\top {W_{1,\delta}}^{-1}$ as the parameter $\delta$ varies, where
\begin{equation*}
	W_\delta=\begin{bmatrix}
		W_{1,\delta} & W_{2,\delta}\\
		W_{2,\delta}^\top & W_{3,\delta}
	\end{bmatrix}
\end{equation*}
is the solution obtained by the solver Pajarito for problem \eqref{misdp}.

We now use PALM Algorithm \ref{PALM-code} to study this example through solving optimization problem \eqref{Pen_1} with parameters:
\begin{equation*}
	\sigma =\frac{1}{50},~~ \beta = 6618,~~ \mu = 1309,~~ \tau = 10454,
\end{equation*}
and it is easy to verify that the above parameters satisfy selection criterion \eqref{chosepara}.
 Letting $\{\widetilde{W}_k,\widetilde{P}_k, z_k, u_k\}_{k\geq 0}$ be the iterative sequence, we begin by verifying the feasibility of the solution produced by Algorithm \ref{PALM-code}, while it suffices to show $\Vert\mathcal{A} \widetilde{W}_k-\mathcal{B}\widetilde{P}_k\Vert \to 0$. Denoting $W_k=\mathrm{vec}^{-1}(\widetilde{W}_k)$, $P_k=\mathrm{vec}^{-1}(\widetilde{P}_k)$ and $W_{2,k}^\top = V_1 W_k V_2^\top$, we  select the initial point $(\widetilde{W}_0,\widetilde{P}_0, z_0, u_0)=(\mathbf{1}_{25\times 1},\mathbf{1}_{6\times 1},\mathbf{1}_{25\times 1},\mathbf{1}_{6\times 1})$ (this “all-ones’’ starting point does not incorporate any prior knowledge and is not biased toward the optimal solution). It is observed that, despite employing a penalty function method, the iterates begin to satisfy the affine constraint $\mathcal{A}\widetilde{W}+\mathcal{B}\widetilde{P}=0$ after approximately $300$ iterations; see Figure \ref{PALM-fig}. The resulting $W$ is given by
	\begin{equation*}
		W=\begin{bmatrix}
			1.428 & -0.940 & 0 & 0.722 & 0.260\\
			-0.940 & 2.440 & 0 & 1.228 & 0.733\\
			0 & 0 & 1.478 & 0 & 1.278 \\
			0.722 & 1.228 & 0 & 1.958 &0.976\\
			0.260 & 0.733 & 1.278 & 0.976 &1.601
		\end{bmatrix},
	\end{equation*}	
	and the resulting optimal feedback gain $K_1$ is given by
	\begin{equation*}
		K_1=\begin{bmatrix}
			1.121 & 0.935 & 0\\
			0.508 & 0.496 & 0.865
		\end{bmatrix}
	\end{equation*}
	with LQ cost $J(K_1)=1.428$. The response of all the state variables is illustrated in Figure \ref{response}, and it demonstrates that the stability of the closed-loop system is guaranteed.

Moreover, we use ADMM algorithm \eqref{ADMM-alg} to investigate this SF-LQ problem by solving problem \eqref{sparse_feedback_6}. Again, denoting $\{u_k, \widetilde{W}_k, \lambda_k, \widetilde{P}_k\}_{k\geq 0}$ as the iterative sequence of ADMM, we select the initial point $(u_0, \widetilde{W}_0, \lambda_0, \widetilde{P}_0)=(\mathbf{0}_{8\times 1}, \mathbf{50}_{25\times 1}, \mathbf{0}_{8\times 1}, \mathbf{0}_{6\times 1})$ (the initial point $\widetilde{W}_0$ is deliberately chosen to be far from optimal, in order to test the robustness of ADMM) and $\beta=10, \xi=0.5$.
As before, we start by verifying the feasibility of the solution. To avoid redundancy and due to space constraints, we only present the trajectory of $[W_k]_{23}$ during iterations; see Figure \ref{ADMM-l0}.
We observe that $\{[W_k]_{23}\}_{k\geq 0}$ quickly approaches $0$ and exhibits regular oscillations around it.
Moreover, we can observe that
\begin{equation*}
	\limsup_{k} [W_k]_{23}=0.5,~~\liminf_{k} [W_k]_{23}=-0.9,~~[-0.9,0.5]=\mathbf{Cluster}(\{[W_k]_{23}\}_{k\geq 0}).
\end{equation*}
This behavior aligns with the theoretical properties, as we can only guarantee that $0$ is a cluster point of $\{[W_k]_{23}\}_{k\geq 0}$, rather than that $\{[W_k]_{23}\}_{k\geq 0}$ converges to $0$.
Additionally, the trajectory of $[W_k]_{13}$ exhibits a similar oscillatory behavior to that observed in $W_k$, suggesting that they also accumulate around $0$.
However, the trajectory of $\Vert P_k-W_{2,k}^\top\Vert$ exhibits a small deviation from the origin; see Figure \ref{ADMM-l0-10}.
Interestingly, by selecting a larger $\beta$ (e.g., $\beta$=100), it holds that $0\in \mathbf{Cluster}(\{\Vert P_k-W_{2,k}^\top\Vert\}_{k\geq 0})$, and the oscillation amplitude of the sequences $\{[W_k]_{13}\}_{k\geq 0}, \{[W_k]_{23}\}_{k\geq 0}, \{\Vert P_k-W_{2,k}^\top\Vert\}_{k\geq 0}$ is significantly reduced. Furthermore, through selecting $\beta=300$, the sequences $\{[W_k]_{13}\}_{k\geq 0}, \{[W_k]_{23}\}_{k\geq 0}, \{\Vert P_k-W_{2,k}^\top\Vert\}_{k\geq 0}$ converge to $0$; see Figure \ref{ADMM-l0-300}.
As mentioned above, ADMM with a larger $\beta$ exhibits better convergence, though, at the expense of sacrificing the group-sparsity level of the resulting $\widetilde{P}$ (or $W_2$). That is because we solve the subproblem
\begin{equation*}
	\widetilde{P}_{n+1}\in\underset{\widetilde{P}}{\operatorname*{argmin}}~\mathcal{L}_\beta(\widetilde{W}_{n+1},\widetilde{P},\lambda_{n+1})~\Longleftrightarrow~
	\widetilde{P}_{n+1} \in \mathbf{prox}_{\frac{\gamma}{\beta}\Vert\cdot\Vert_{s,t;0}}(z)
\end{equation*}
within the ADMM framework, where $z=-\left[\mathcal{A}\widetilde{W}_{n+1}+\frac{1}{\beta}\lambda_{n+1}\right]_{[3:8]}$ (the slice of the vector from index $3$ to $8$). In addition, the solution of the above proximal operator tends to become less sparse as $\beta$ increases.
Hence, for the individual case, we need to make a case-by-case study so that a proper $\beta$ can be found, which is essentially a balance between the oscillation amplitude and group-sparsity level.
Remarkably, even though there is a loss of convergence, by selecting a subsequence from the iterative sequence, we can still recover the group-sparse feedback gain $K$.
For instance, letting $\beta = 30$, we may choose a proper subsequence $\mathbb{K}\subseteq \mathbb{N}$ such that $\{W_k\}_{k\in\mathbb{K}}$ converges to
\begin{equation*}
	W = \begin{bmatrix}
		6.346 & -2.709 & 0 & 0.200 & 0\\
		-2.709 & 1.172 & 0 & 1.115 & 0\\
		0 & 0 & 0.759 & 0 & 1.026\\
		0.200 & 1.115 & 0 & 7.861 & 3.460\\
		0 & 0 & 1.026 & 3.460 & 2.615
	\end{bmatrix}
\end{equation*}
and the feedback gain is given by
\begin{equation*}
	K_2=\begin{bmatrix}
		32.934 & 77.077  & 0       \\
 0       &  0        &  1.354
	\end{bmatrix}	
\end{equation*}
with the LQ cost $J(K_2)=6.346$.

Additionally, we use ADMM algorithm \eqref{ADMM-alg} to solve the following group-$\ell_1$ relaxation of SF-LQ problem:
\begin{equation*}
\begin{aligned}
\min_{\widetilde{W},\widetilde{P}}~~ &f(\widetilde{W})+\gamma\Vert \mathrm{vec}^{-1}(\widetilde{P}) \Vert_{s,t;1}\\
{\rm s.t.}~~&\mathcal{A}\widetilde{W}+\mathcal{B}\widetilde{P}=0
\end{aligned}
\end{equation*}
with $\gamma=50$;
we refer to \cite{feng2024optimization} for the details of algorithm. The feedback gain is given by
\begin{equation*}
	K_3=\begin{bmatrix}
65.688 & 169.089 & 0.284 \\
0 & 0 & 1.095
\end{bmatrix}
\end{equation*}
with LQ cost $J(K_3)=5.855$.
Moreover, letting $\gamma = 200$, we may obtain the following feedback gain
\begin{equation*}
	K_4=\begin{bmatrix}
5.489 & 1.946 & 0 \\
0 & 0 & 1.815
\end{bmatrix}
\end{equation*}
with LQ cost $J(K_4)=9.173$.
We may observe that, although $K_1$ and $K_3$ (or $K_2$ and $K_4$) share the same group $\ell_0$-norm w.r.t. block structure \eqref{block-structure}, $J(K_3)$ is over four times larger than $J(K_1)$, since $\Vert \mathrm{vec}^{-1}(\widetilde{P}) \Vert_{s,t;1}$ dominates the LQ cost $f(\widetilde{W})$ as mentioned in Section \ref{contribution-sec}.  That is why we study the SF-LQ problem from a nonconvex optimization perspective. 
\end{solution}

\begin{example}\label{exp2}
	Consider $x=[x_1,x_2,x_3, x_4, x_5]^\top$ and a relatively larger scale linear system with
	\begin{equation*}
  \begin{aligned}
  &A=\begin{bmatrix}
  0.3079 & 0.1879 & 0.1797 & 0.2935 & 0.6537\\
  0.5194 & 0.2695 & 0.5388 & 0.9624 & 0.5366\\
  0.7683 & 0.4962 & 0.2828 & 0.9132 & 0.9957\\
  0.7892 & 0.7391 & 0.7609 & 0.5682 & 0.1420\\
  0.8706 & 0.1950 & 0.2697 & 0.4855 & 0.9753
  \end{bmatrix},B_2=\begin{bmatrix}
  0.6196 & 0.6414\\
  0.7205 & 0.9233\\
  0.2951 & 0.8887\\
  0.6001 & 0.6447\\
  0.7506 & 0.2956
  \end{bmatrix}\\
  & B_1=I_5,C=\begin{bmatrix}
  1 & 0 & 0 & 0 & 0\\
  0 & 1 & 0 & 0 & 0\\
  0 & 0 & 0 & 0 & 0\\
  0 & 0 & 0 & 0 & 0\\
  0 & 0 & 0 & 0 & 0
  \end{bmatrix},D=\begin{bmatrix}
  0 & 0\\
  0 & 0\\
  1 & 0\\
  0 & 1\\
  0 & 0
  \end{bmatrix}
  \end{aligned}
\end{equation*}
and the noise $w$ is characterized by a impulse disturbance vector.
\end{example}

\begin{solution}
We consider feedback gain $K$ with the following block structure:
\begin{equation*}
	K = \mathbf{block}_{2,3}(K_{11},K_{12},K_{13},K_{21},K_{22},K_{23})
\end{equation*}
with $K_{11},K_{12},K_{21},K_{22}\in\mathbb{R}^{1\times 2}$ and $K_{13},K_{23}\in\mathbb{R}$.
Let 
$$\mathcal{N}=\{(1,3),(1,4),(1,5),(2,3),(2,4),(2,5),(3,5),(4,5)\}$$
 and $\{u_k, \widetilde{W}_k, \lambda_k, \widetilde{P}_k\}_{k\geq 0}$ be the iterative sequence of PALM for solving problem \eqref{Pen_1}, and obviously, we hope $\sum_{(i,j)\in\mathcal{N}}|W_{k,ij}|\to 0$. We select the initial point $(\widetilde{W}_0,\widetilde{P}_0, z_0, u_0)=(\mathbf{1}_{49\times 1},\mathbf{1}_{10\times 1},\mathbf{1}_{49\times 1},\mathbf{1}_{10\times 1})$, while the following Figure \ref{ADMM-large} demonstrate the asymptotic feasibility of PALM. The resulting optimal feedback gain $K$ is given by
\[K=
\begin{bmatrix}
1.072 & -0.211 & 0.017 & -0.434 & 3.825 \\
-0.310 & 0.595 & 0.695 & 2.047 & 0
\end{bmatrix},
\]
while it can be readily verified that the closed-loop system is stable (the system response curves are omitted for brevity).

We then consider the DFT-LQ problem for Example \ref{exp2}, where we force $K_{22}=\mathbf{0}_{1\times 2}$. Based on the discussions in Section \ref{section_link}, the resulting optimal feedback gain $K$ is given by
\[
K = \begin{bmatrix}
1.449 & 0.208 & 2.855 & 4.266 & 1.961 \\
-0.375 & 0.430 & 0 & 0 & -0.987
\end{bmatrix}.
\]
Alternatively, if we force $K_{13}=0$, then the resulting optimal feedback gain $K$ is given by
\[
K = \begin{bmatrix}
0.713 & -0.950 & -0.268 & -0.411 & 0 \\
-0.022 & 1.219 & 0.890 & 1.482 & 9.626
\end{bmatrix}.
\]
Hence, while the studies of this paper center on the SF-LQ problem, the developed framework and methodology are sufficiently general to accommodate the DFT-LQ case as well.
\end{solution}

\begin{figure}[htbp]
    \centering
    \begin{subfigure}[b]{0.45\textwidth}
        \includegraphics[width=\textwidth]{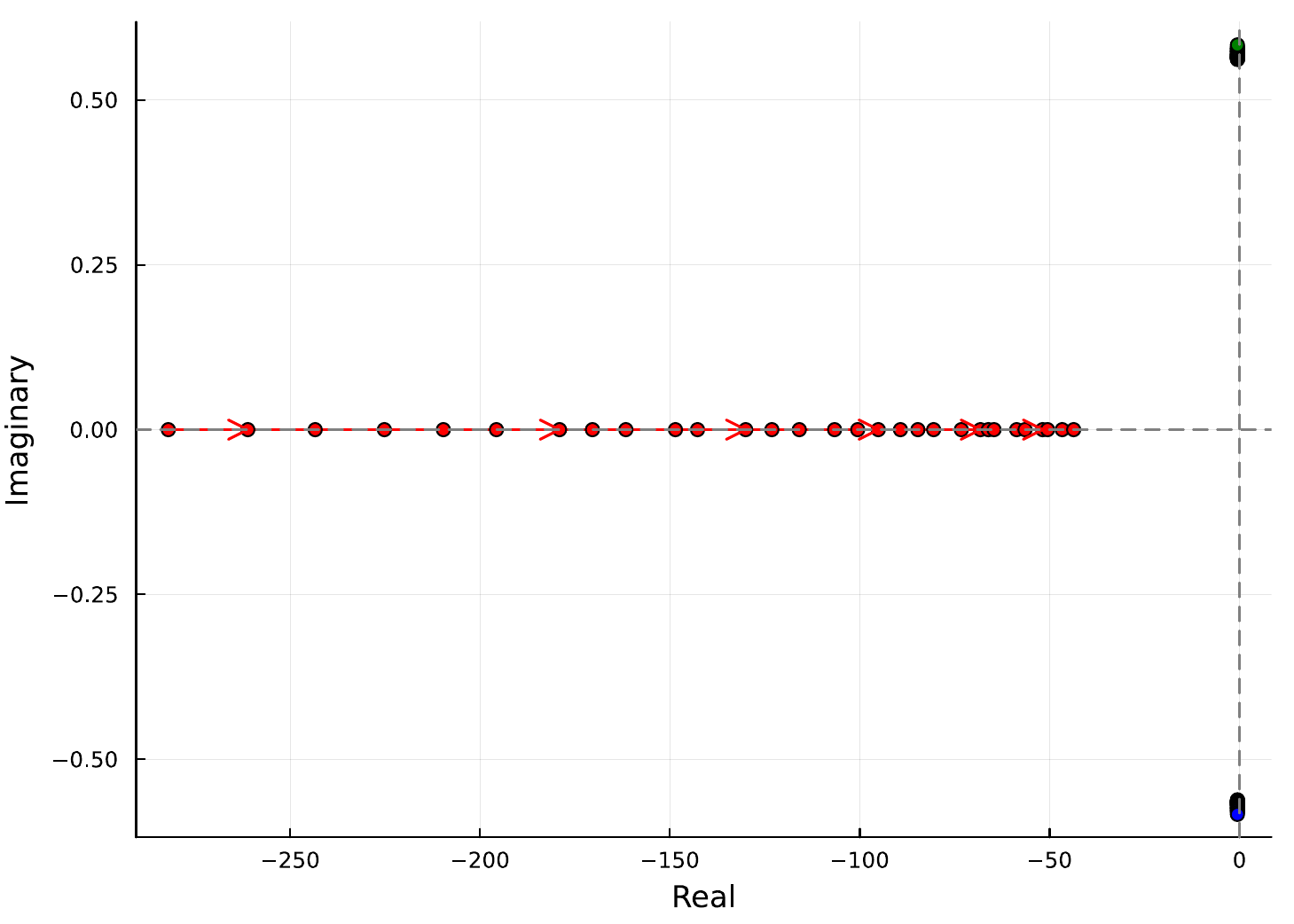}
        \caption{Eigenvalue Trajectories as $\delta\in[0.01,0.1]$ Varies}
        \label{eigenvalue_trajactory}
    \end{subfigure}
    \hfill
    \begin{subfigure}[b]{0.45\textwidth}
        \includegraphics[width=\textwidth]{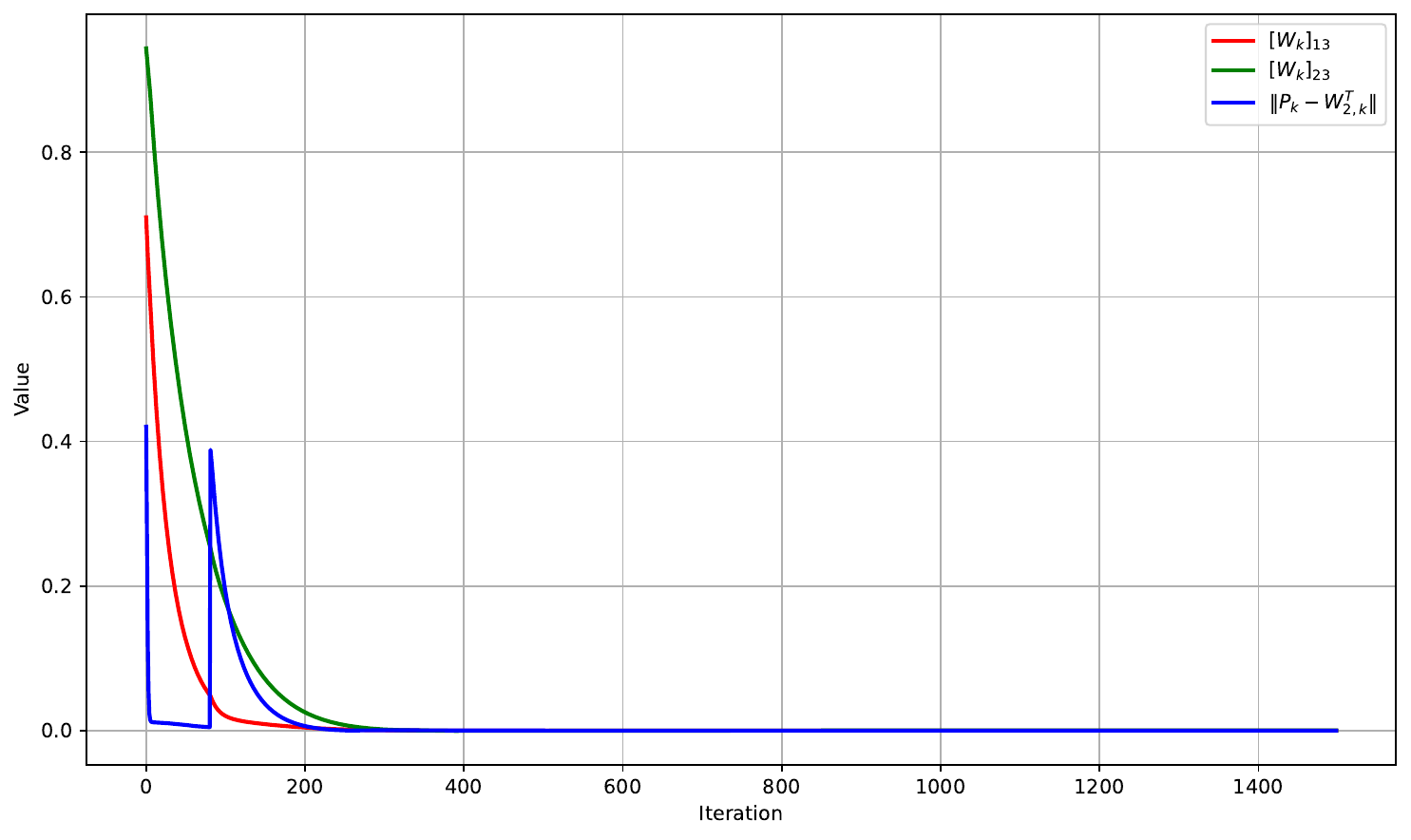}
        \caption{Feasibility of PALM for Affine Constraint $\mathcal{A}\widetilde{W}+\mathcal{B}\widetilde{P}=0$}
        \label{PALM-fig}
    \end{subfigure}

    \vspace{0.3cm}

    \begin{subfigure}[b]{0.45\textwidth}
        \includegraphics[width=\textwidth]{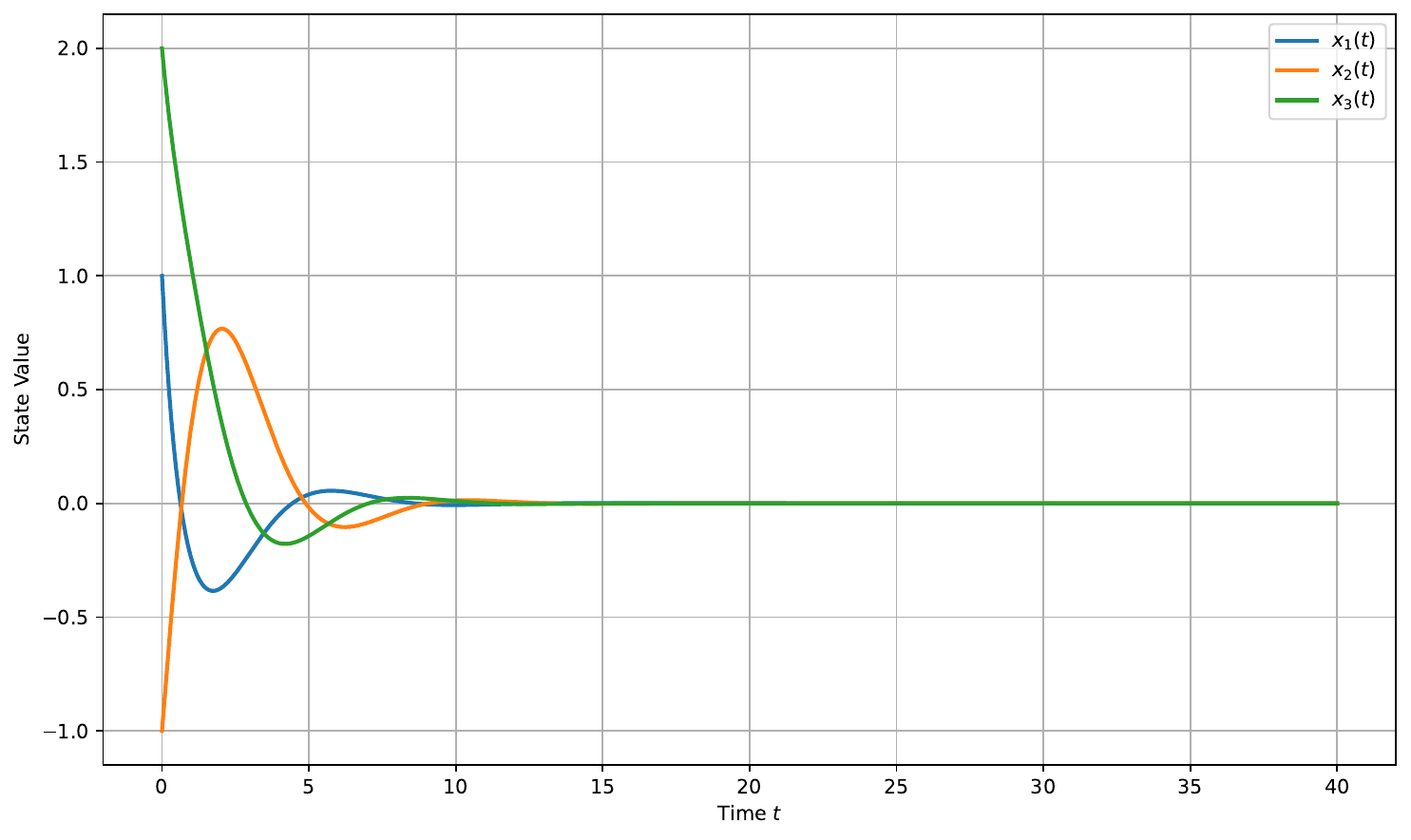}
        \caption{System Response}
        \label{response}
    \end{subfigure}
    \hfill
    \begin{subfigure}[b]{0.45\textwidth}
        \includegraphics[width=\textwidth]{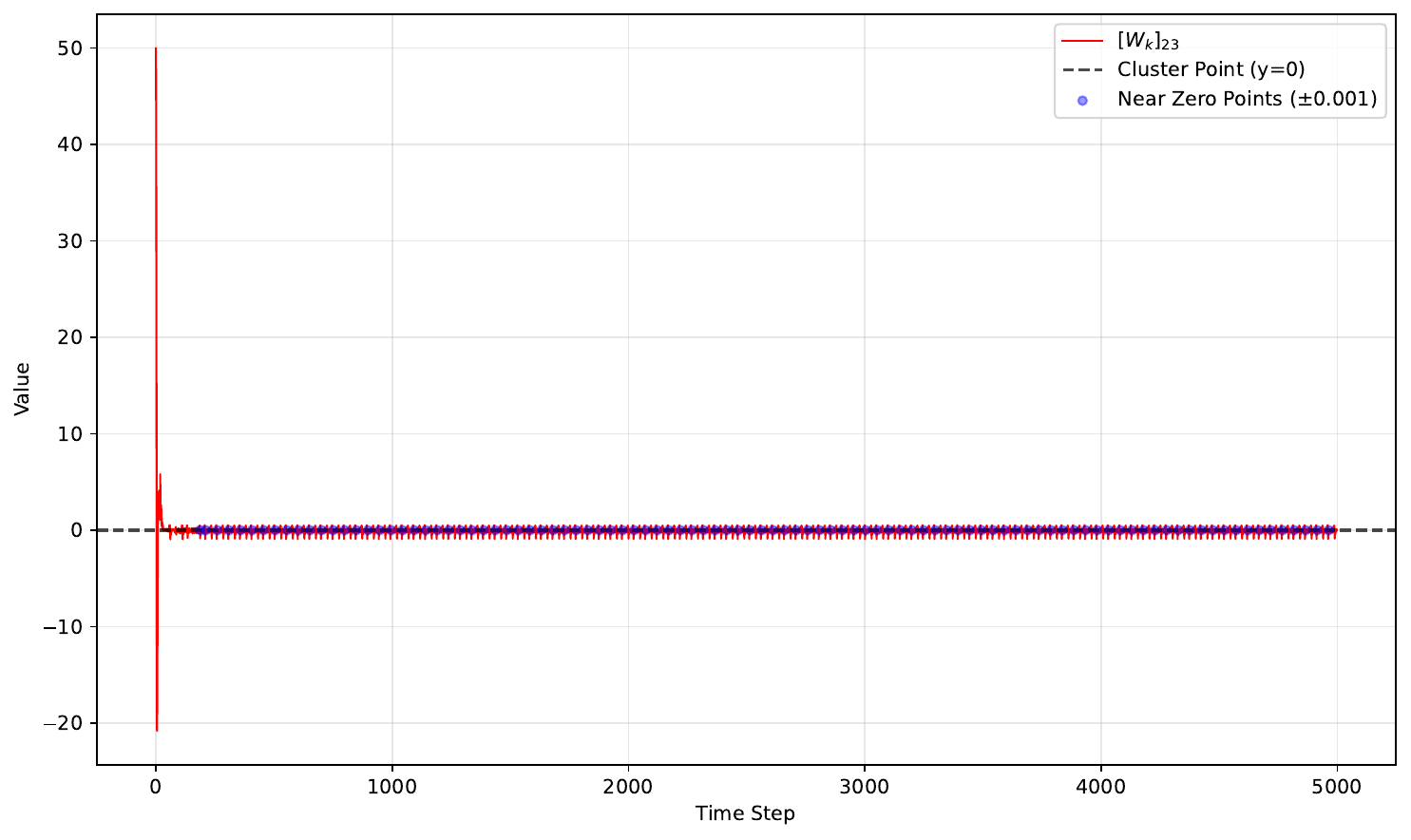}
        \caption{Trajectory of $[W_k]_{23}$}
        \label{ADMM-l0}
    \end{subfigure}

    \vspace{0.3cm}

    \begin{subfigure}[b]{0.45\textwidth}
        \includegraphics[width=\textwidth]{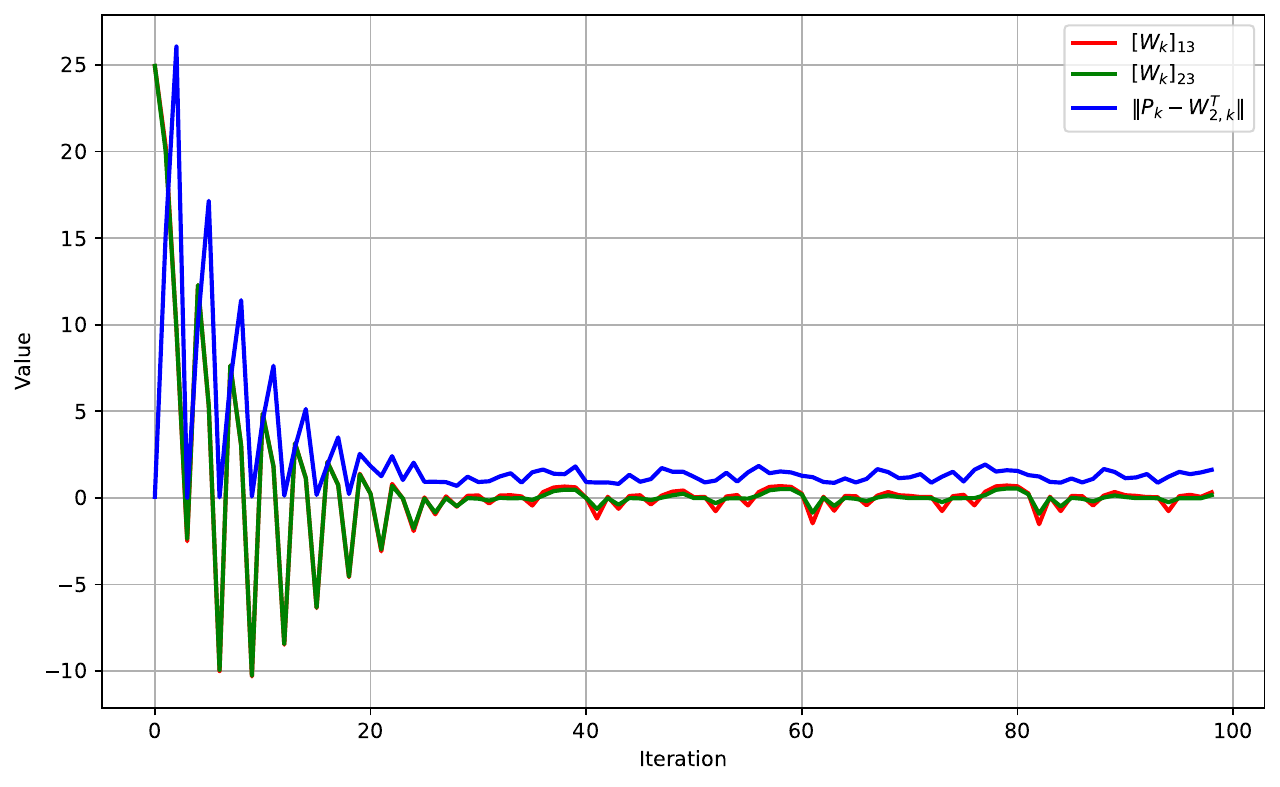}
        \caption{Non-Feasibility of $\Vert P_k-W_{2,k}^\top\Vert$}
        \label{ADMM-l0-10}
    \end{subfigure}
    \hfill
    \begin{subfigure}[b]{0.45\textwidth}
        \includegraphics[width=\textwidth]{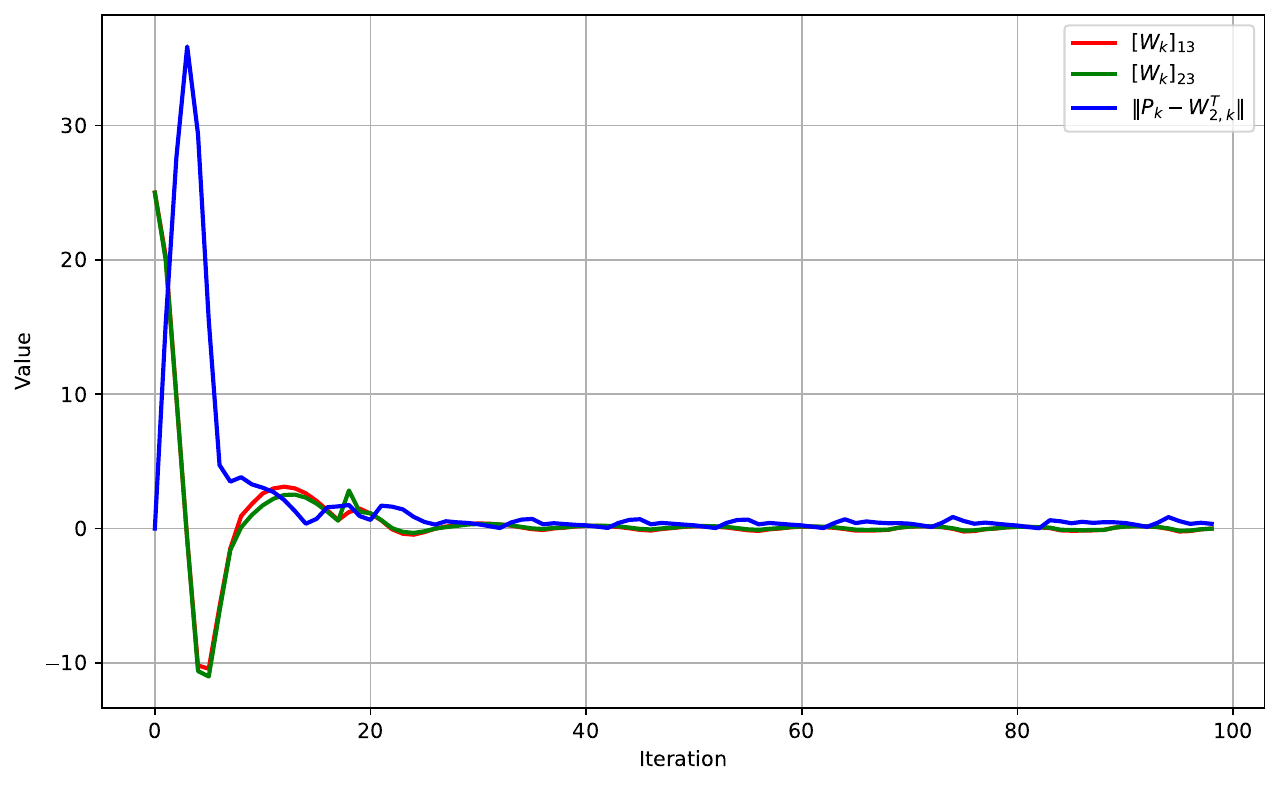}
        \caption{Subsequence Asymptotic Feasibility of $\Vert P_k-W_{2,k}^\top\Vert$ when $\beta=100$}
        \label{ADMM-l0-100}
    \end{subfigure}

    \vspace{0.3cm}

    \begin{subfigure}[b]{0.45\textwidth}
        \includegraphics[width=\textwidth]{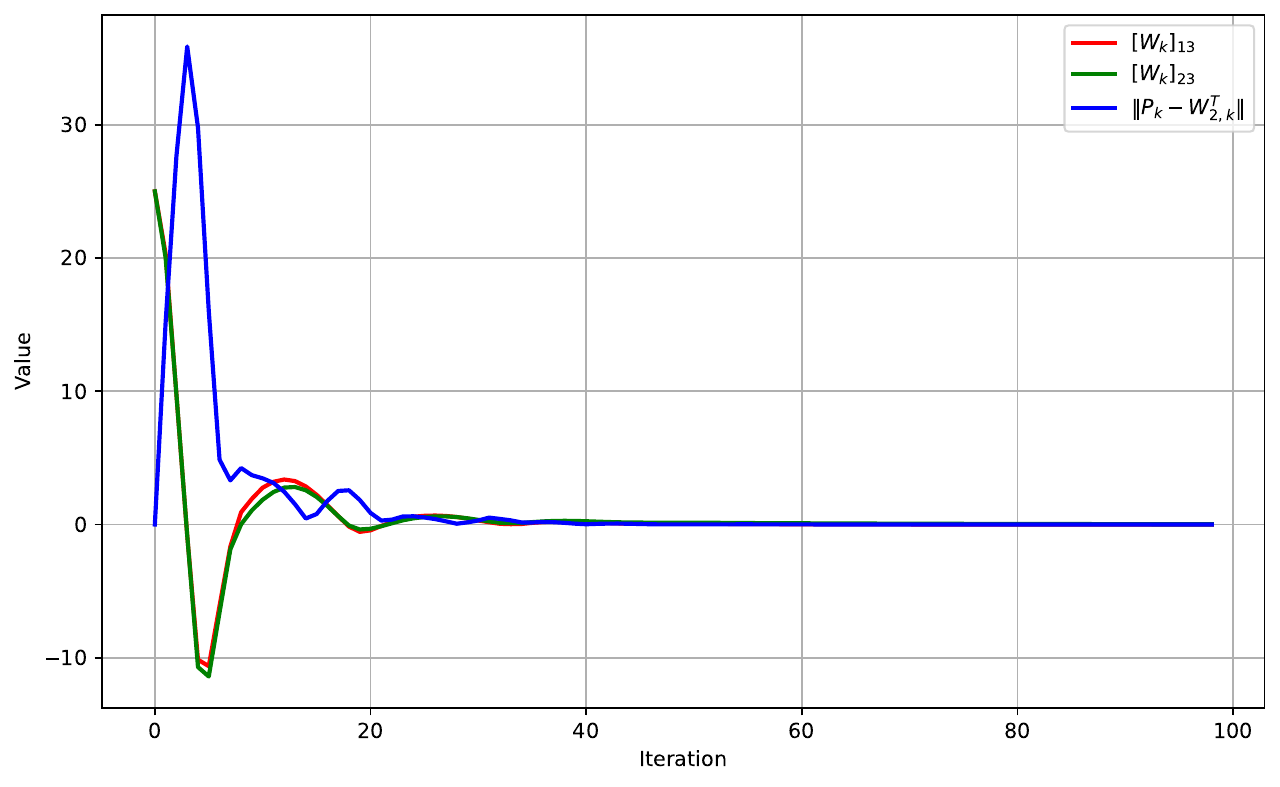}
        \caption{Asymptotic Feasibility when $\beta=300$}
        \label{ADMM-l0-300}
    \end{subfigure}
    \hfill
    \begin{subfigure}[b]{0.45\textwidth}
        \includegraphics[width=\textwidth]{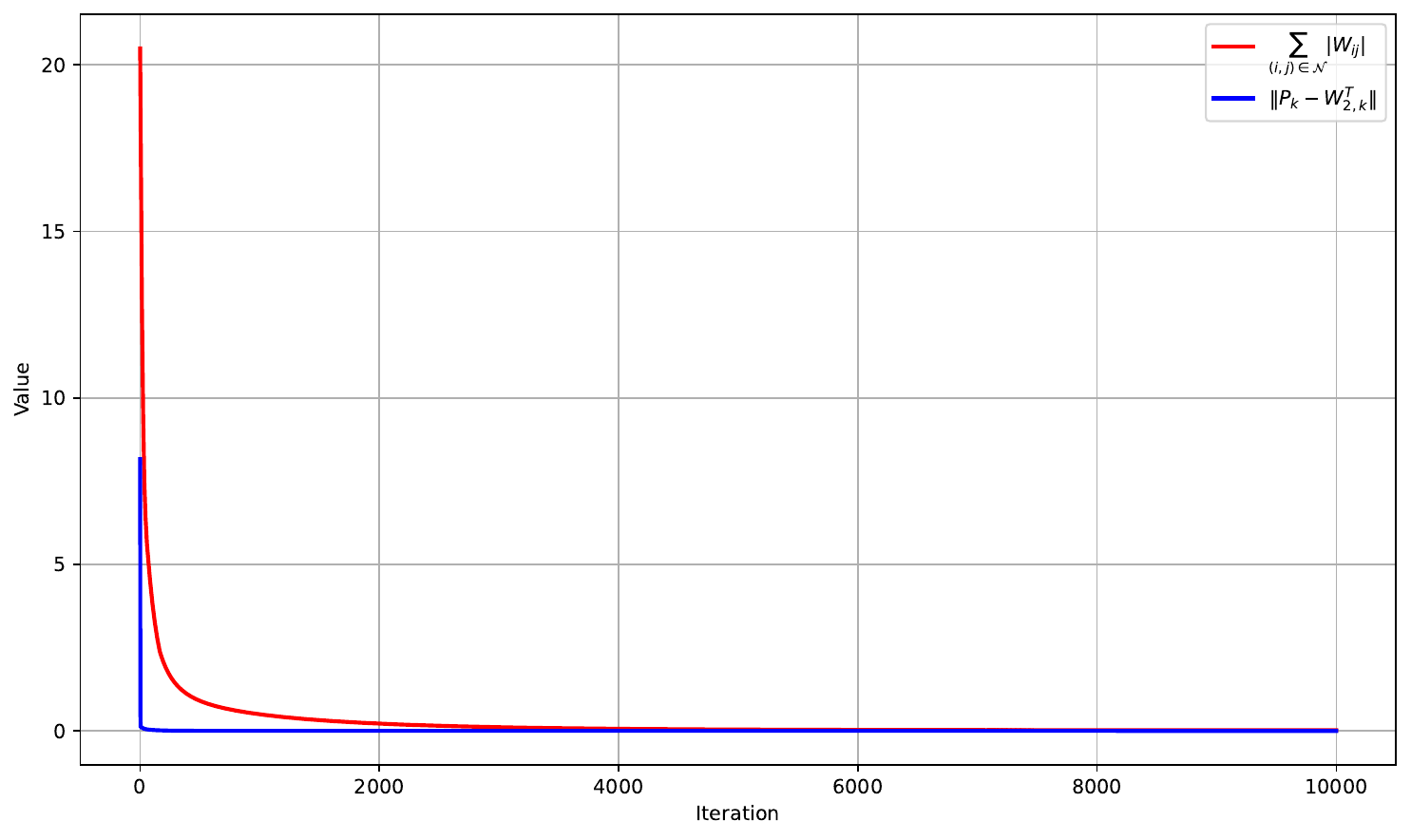}
        \caption{Feasibility of PALM for Example \ref{exp2}}
        \label{ADMM-large}
    \end{subfigure}
    \caption{Figures of Section \ref{section6}}
    \label{union1}
\end{figure}

\section{Conclusion}

This paper presents a unified nonconvex optimization framework for designing group-sparse feedback controllers in infinite-horizon LQ problems. Our main contributions are threefold: First, we establish a connection between the DFT-LQ problem and the SF-LQ problem, showing they can be addressed within the same optimization framework. Second, we develop a PALM algorithm that efficiently solves the penalized nonconvex optimization problem \eqref{contri-2}, providing rigorous convergence analysis. Our framework fills an important gap in the literature by enabling direct design of group-sparse feedback gains with theoretical guarantees, without relying on convex surrogates or restrictive structural assumptions. In the near future, we may extend our framework to the case of time-varying communication topology and the case of time-varying system matrices.

\section*{Appendix}

\subsection*{Proof of Theorem \ref{boundness-thm}}

\begin{proof}
For a fixed $n \geq 1$, by Lemma \ref{selectrule}, we know that $C_2, C_3, C_4 > 0$, and from the parameter selection criterion \eqref{chosepara}, we know that $2\tau > \beta$. Therefore, Lemma \ref{desecent_lemma} is applicable. From \eqref{descent-chap3}, we know that
\begin{equation*}
  \begin{aligned}
  \Psi_1&\geq\cdots\geq\Psi_n\geq\Psi_{n+1}\\
  &\geq f(z_{n+1})+g(\widetilde{P}_{n+1})+H(\widetilde{W}_{n+1},\widetilde{P}_{n+1})+\langle u_{n+1},\widetilde{W}_{n+1}-z_{n+1}\rangle+\frac{\beta}{2}\Vert \widetilde{W}_{n+1}-z_{n+1}\Vert^2\\
  &=f(z_{n+1})+g(\widetilde{P}_{n+1})+H(\widetilde{W}_{n+1},\widetilde{P}_{n+1})-\frac{1}{2\beta}\Vert u_{n+1}\Vert^2+\frac{\beta}{2}\left\Vert \widetilde{W}_{n+1}-z_{n+1}+\frac{1}{\beta}u_{n+1}\right\Vert^2.
  \end{aligned}
\end{equation*}
By \eqref{PAML}, it is straightforward to deduct that
\begin{equation*}
  \begin{aligned}
  u_{n+1}&=u_n+\sigma\beta(\widetilde{W}_{n+1}-z_{n+1})\\
  &=u_{n+1}+(\sigma-1)\beta(\widetilde{W}_{n+1}-z_{n+1})+\beta(\widetilde{W}_{n+1}-z_{n+1})\\
  &=\left(1-\frac{1}{\sigma}\right)(u_{n+1}-u_n)+u_n+\beta(\widetilde{W}_n-z_{n+1})+\beta(\widetilde{W}_{n+1}-\widetilde{W}_n)\\
  &=\left(1-\frac{1}{\sigma}\right)(u_{n+1}-u_n)+(\tau-\beta-\rho \mathcal{A}^\top\mathcal{A})(\widetilde{W}_{n}-\widetilde{W}_{n+1})\\
  &~~~~-\rho \mathcal{A}^\top (\mathcal{A}\widetilde{W}_{n+1}+\mathcal{B}\widetilde{P}_{n+1}).
  \end{aligned}
\end{equation*}
Letting
\begin{equation}\label{B*}
  B_*\doteq \sup\limits_{n\geq 0}\{\Vert \widetilde{W}_{n+1}-\widetilde{W}_n\Vert,\Vert \widetilde{P}_{n+1}-\widetilde{P}_n\Vert,\Vert z_{n+1}-z_n\Vert,\Vert u_{n+1}-u_n\Vert\}
\end{equation}
and based on Theorem 2.8 of \cite{bot2019proximal}, it holds that $B_*<\infty$; hence it follows that
\begin{equation*}
  \Vert u_{n+1}\Vert\leq \left(\frac{1}{\sigma}-1+\tau+\kappa_1\right)B_*+\rho\Vert \mathcal{A}^\top (\mathcal{A}\widetilde{W}_{n+1}+\mathcal{B}\widetilde{P}_{n+1})\Vert.
\end{equation*}
Using Cauchy-Schwarz inequality, we have
\begin{equation*}
  \Vert u_{n+1}\Vert^2\leq 2\left(\frac{1}{\sigma}-1+\tau+\kappa_1\right)^2B_*^2+2\rho^2\Vert \mathcal{A}^\top (\mathcal{A}\widetilde{W}_{n+1}+\mathcal{B}\widetilde{P}_{n+1})\Vert^2.
\end{equation*}
In summary, we obtain an estimate for the lower bound of $\Psi_1$:
\begin{equation}\label{lowerboundPsi}
\begin{aligned}
  \Psi_1\geq&f(z_{n+1})+g(\widetilde{P}_{n+1})+H(\widetilde{W}_{n+1},\widetilde{P}_{n+1})+\frac{\beta}{2}\left\Vert \widetilde{W}_{n+1}-z_{n+1}+\frac{1}{\beta}u_{n+1}\right\Vert^2\\
  &-\frac{1}{\beta}\left(\frac{1}{\sigma}-1+\tau+\kappa_1\right)^2B_*^2-\frac{\rho^2}{\beta}\Vert \mathcal{A}^\top (\mathcal{A}\widetilde{W}_{n+1}+\mathcal{B}\widetilde{P}_{n+1})\Vert^2.
\end{aligned}
\end{equation}
On one hand, by the definitions of $f$, $g$ and for any $n \geq 1$, we have
\begin{equation}\label{upper1}
\begin{aligned}
&\frac{1}{2}H(\widetilde{W}_{n+1},\widetilde{P}_{n+1})+\frac{\beta}{2}\left\Vert \widetilde{W}_{n+1}-z_{n+1}+\frac{1}{\beta}u_{n+1}\right\Vert^2\\
\leq&\Psi_1+\frac{1}{\beta}\left(\frac{1}{\sigma}-1+\tau+\kappa_1\right)^2B_*^2\\
&-\frac{1}{2}\inf\limits_{n\geq1}\left\{H(\widetilde{W}_{n+1},\widetilde{P}_{n+1})-\left(\frac{1}{\gamma_2}-\frac{\kappa_1}{2\gamma_2^2}\right)\rho^2\Vert \mathcal{A}^\top (\mathcal{A}\widetilde{W}_{n+1}+\mathcal{B}\widetilde{P}_{n+1})\Vert^2\right\} \\
\overset{\mathrm{(\romannumeral1)}}{\leq}&\Psi_1+\frac{1}{\beta}\left(\frac{1}{\sigma}-1+\tau+\kappa_1\right)^2B_*^2 < \infty. 
\end{aligned}
\end{equation}
Here, based on the $\kappa_1$-Lipschitz continuity of $H$ with respect to $\widetilde{W}$, inequality (\romannumeral1) holds. On the other hand,
\begin{equation}\label{upper2}
  \begin{aligned}
  &f(z_{n+1})+g(\widetilde{P}_{n+1})+\frac{\beta}{2}\left\Vert \widetilde{W}_{n+1}-z_{n+1}+\frac{1}{\beta}u_{n+1}\right\Vert^2\\
  \leq & \Psi_1+\frac{1}{\beta}\left(\frac{1}{\sigma}-1+\tau+\kappa_1\right)^2B_*^2\\
  &-\inf\limits_{n\geq1}\left\{H(\widetilde{W}_{n+1},\widetilde{P}_{n+1})-\left(\frac{1}{\gamma_1}-\frac{\kappa_1}{2\gamma_1^2}\right)\rho^2\Vert \mathcal{A}^\top (\mathcal{A}\widetilde{W}_{n+1}+\mathcal{B}\widetilde{P}_{n+1})\Vert^2\right\}\\
  \leq& \Psi_1+\frac{1}{\beta}\left(\frac{1}{\sigma}-1+\tau+\kappa_1\right)^2B_*^2<\infty.
  \end{aligned}
\end{equation}
By \eqref{upper1}, we can deduce that the sequence $\{\widetilde{W}_{n}-z_{n}+u_{n}\}_{n\geq 0}$ is bounded, while by \eqref{PAML}, we have $u_{n+1}-u_n=\sigma\beta (\widetilde{W}_{n+1}-z_{n+1})$.
Therefore, the sequence $\{\widetilde{W}_{n} - z_{n}\}_{n \geq 0}$ is bounded; otherwise, it would contradict $B_* < \infty$; and the fact above implies that the sequence $\{u_n\}_{n \geq 0}$ is also bounded. 
By Assumption \ref{ass1}, we have
$$f(z_{n+1})\geq \langle \mathrm{vec}(R), z_{n+1} \rangle\geq \lambda_{\rm min}(R)\mathrm{Tr}(\mathrm{vec}^{-1}(z_{n+1}))\geq \lambda_{\rm min}(R)\Vert \mathrm{vec}^{-1}(z_{n+1})\Vert_2$$
with $\lambda_{\rm min}(R)>0$. Hence, combining with \eqref{upper2}, we can deduce that $\{z_n\}_{n\geq 0}$ is bounded. Due to the boundedness of sequences $\{z_n\}_{n\geq 0}$ and  $\{\widetilde{W}_{n} - z_{n}\}_{n \geq 0}$, $\{\widetilde{W}_n\}_{n \geq 0}$ is bounded. 
By further using \eqref{upper1}, we have
\begin{equation*}
  \infty>H(\widetilde{W}_n,\widetilde{P}_n)=\frac{\rho}{2}\Vert\mathcal{A}\widetilde{W}_n+\mathcal{B}\widetilde{P}_n \Vert^2\geq \Vert \mathrm{vec}(W_{2,n})-\widetilde{P}_n\Vert^2
\end{equation*}
with $W_{2,n}=V_1\mathrm{vec}^{-1}(\widetilde{W}_n)V_2^\top$.
Since $\{\widetilde{W}_n\}_{n \geq 0}$ is bounded, it follows that $\{W_{2,n}\}_{n \geq 0}$ is bounded, and consequently, $\{\widetilde{P}_n\}_{n \geq 0}$ is bounded. In summary, we have proven that the iterative sequence $\{(\widetilde{W}_n, \widetilde{P}_n, z_n, u_n)\}_{n \geq 0}$ is bounded.
\end{proof}

\bibliographystyle{plain}
\bibliography{tac_7.9}

\end{document}